\documentclass{article}

\usepackage[utf8]{inputenc}
\usepackage[english]{babel}

\usepackage{geometry}
\usepackage{fancyhdr}
\usepackage{titlesec}
\usepackage{indentfirst}
\usepackage{float}

\usepackage{amsmath}
\usepackage{amssymb}
\usepackage{mathrsfs}
\usepackage{amsthm}
\usepackage{amsfonts}

\usepackage[linktoc=all]{hyperref}
\usepackage{cleveref}

\usepackage{tikz}

\usepackage{array}
\newcolumntype{C}[1]{>{\centering\arraybackslash}p{#1}}

\theoremstyle{plain}
\newtheorem{thm}{Theorem}[section]
\newtheorem{prop}[thm]{Proposition}
\newtheorem{lem}[thm]{Lemma}
\newtheorem{cor}[thm]{Corollary}
\newtheorem{defn}[thm]{Definition}

\newtheorem{introthm}{Theorem}

\theoremstyle{remark}
\newtheorem{ex}[thm]{Example}
\newtheorem{rmk}[thm]{Remark}

\DeclareMathOperator{\rank}{rk}

\newcommand{\bA}{\mathbf{A}}

\newcommand{\bZ}{\mathbf{Z}}

\newcommand{\ba}{\mathbf{a}}
\newcommand{\bb}{\mathbf{b}}

\newcommand{\be}{\mathbf{e}}
\newcommand{\bh}{\mathbf{h}}
\newcommand{\bl}{\mathbf{l}}
\newcommand{\bm}{\mathbf{m}}
\newcommand{\bp}{\mathbf{p}}
\newcommand{\bq}{\mathbf{q}}
\newcommand{\br}{\mathbf{r}}
\newcommand{\bs}{\mathbf{s}}
\newcommand{\bt}{\mathbf{t}}
\newcommand{\bu}{\mathbf{u}}
\newcommand{\bv}{\mathbf{v}}
\newcommand{\bw}{\mathbf{w}}
\newcommand{\bx}{\mathbf{x}}
\newcommand{\by}{\mathbf{y}}
\newcommand{\bz}{\mathbf{z}}
\newcommand{\bepsilon}{\mathbf{\epsilon}}

\newcommand{\bbN}{\mathbb{N}}

\newcommand{\bbR}{\mathbb{R}}
\newcommand{\bbZ}{\mathbb{Z}}

\newcommand{\cG}{\mathcal{G}}

\newcommand{\cP}{\mathcal{P}}
\newcommand{\cS}{\mathcal{S}}

\newcommand{\rar}{\rightarrow}

\newcommand{\ol}[1]{\overline{#1}}
\newcommand{\ot}[1]{\widetilde{#1}}
\newcommand{\abs}[1]{\left|#1\right|}
\newcommand{\divides}{\bigm|}
\newcommand{\gen}[1]{\langle #1 \rangle}

\newcommand{\pA}{\mathbf{A}^+}
\newcommand{\supp}[1]{\mathrm{supp}(#1)}

\newcommand{\cnj}{\sim_{\mathrm{c}}}
\newcommand{\lecnj}{\preceq_{\mathrm{c}}}
\newcommand{\qcnj}{\sim_{\mathrm{qc}}}
\newcommand{\qcsupp}[1]{\mathrm{supp}_\mathrm{qc}(#1)}

\newcommand{\xdash}[1][3em]{\rule[0.5ex]{#1}{0.55pt}}
\newcommand{\edgedash}[1]{\ \xdash[#1]\ }
\newcommand{\edge}{\edgedash{0.6cm}}

\newcommand{\FG}[1]{\mathrm{FG}(#1)}

\title{On the isomorphism problem for generalized Baumslag-Solitar groups: angles}

\author{
Dario Ascari\\
{\small \textit{Department of Mathematics, University of the Basque Country,}}\\
{\small \textit{Barrio Sarriena, Leioa, 48940, Spain}}\\
{\small e-mail: \texttt{ascari.maths@gmail.com}}\\
\and
Montserrat Casals-Ruiz\\
{\small \textit{Ikerbasque - Basque Foundation for Science and Matematika Saila,}}\\
{\small \textit{UPV/EHU,  Sarriena s/n, 48940, Leioa - Bizkaia, Spain}}\\
{\small e-mail: \texttt{montsecasals@gmail.com}}\\
\and
Ilya Kazachkov\\
{\small \textit{Ikerbasque - Basque Foundation for Science and Matematika Saila,}}\\
{\small \textit{UPV/EHU,  Sarriena s/n, 48940, Leioa - Bizkaia, Spain}}\\
{\small e-mail: \texttt{ilya.kazachkov@gmail.com}}
}

\date{}

\begin{document}

\maketitle

\begin{abstract}
    We introduce a new isomorphism invariant for generalized Baumslag–Solitar (GBS) groups, which we call the limit angle. Unlike previously known invariants, which are primarily algebraic, the limit angle admits a dynamical interpretation, arising exclusively in the case of two interacting edges. This invariant captures subtle geometric behavior that does not manifest in configurations with more interacting edges. As an application, we use the limit angle to obtain a classification of GBS groups with one vertex and two edges.
\end{abstract}


\section{Introduction}

The isomorphism problem for generalized Baumslag–Solitar groups (GBSs) is a long-standing open question that has attracted considerable attention (see \cite{For06,Lev07,CF08,Dud17,CRKZ21,Wan25}). Interest in this problem stems from the theory of JSJ decompositions, which plays a central role in resolving the isomorphism problem for hyperbolic groups \cite{Sel95,DG11,DT19}. For those groups the JSJ decomposition is essentially unique and thus this naturally leads to the question of whether the JSJ decomposition can be useful to classify broader classes of groups that admit infinitely many different JSJ decompositions. In this context, GBS groups represent the simplest unsolved case. The significance of GBSs is further underscored by recent results of the authors \cite{ACK-iso1,ACK-out}, which provide evidence that the isomorphism problem and the structure of the outer automorphism groups for large families of groups can be reduced to the case of GBS groups.

In \cite{ACK-iso1} the authors initiated a systematic study of the isomorphism problem for GBS groups, see references therein for previously known results. They proved that for any two isomorphic GBS groups, there exists an explicit (computable) bound on the number of vertices and edges in the graph of groups decompositions appearing along a sequence of moves that realizes the isomorphism.

This line of research was continued in \cite{ACK-iso2}, where the authors introduced several new isomorphism invariants for GBS groups (the linear invariants, the set of rigid edges, and the assignment map) and showed that these invariants are sufficient to determine the isomorphism class within a broad family of GBS groups: those with one quasi-conjugacy class, full-support gaps, and at least three interacting edges. 

However, when only two edges interact with each other, the previously introduced invariants are no longer sufficient to fully characterize the isomorphism class. In this paper, we address this case by introducing a new invariant for GBS groups: the \emph{limit angle}. This invariant has a geometric interpretation that captures qualitatively new dynamics, one that arises exclusively in the case of two interacting edges. Furthermore, with the invariants introduced in \cite{ACK-iso2} and the limit angle, we are able to solve the isomorphism problem  for groups with one vertex and two edges:

\begin{introthm}[\Cref{cor:main}]\label{introthm:main}
    There is an algorithm that, given two GBS graphs $(\Gamma,\psi),(\Delta,\phi)$, where $(\Gamma,\psi)$ has one vertex and two edges, decides whether or not the corresponding GBS groups are isomorphic. 
\end{introthm}

We believe that this invariant, together with the main result of \Cref{introthm:main}, will play a crucial role in solving the isomorphism problem for GBS groups.

\bigskip

\paragraph{Description of the invariant.} We consider one-vertex two-edges GBS graphs $(\Gamma,\psi)$ with two minimal points. For each such GBS graph (with few exceptions, see \Cref{rmk:classification-roots-without-limit-directions}), we define two limit directions $\bl^-,\bl^+\in \bbR^2$ (\Cref{def:limit-directions}). The positive cone $P_{(\Gamma,\psi)}:=\{\lambda \bl^-+\mu\bl^+ : \lambda,\mu\in[0,+\infty)\}\subseteq \bbR^2$ in the Euclidean plane is called the \textit{limit angle} of $(\Gamma,\psi)$.

We show that the limit angle, when considered alongside a fixed family of linear invariants and an assignment map, completely determines the isomorphism type of the GBS group.

\begin{introthm}[\Cref{thm:limit-angles}]\label{introthm:limit-angles}
    Let $(\Gamma,\psi),(\Delta,\phi)$ be two one-vertex two-edges GBS graphs with two minimal points, and the same linear invariants and assignment map. Suppose that the limit angles $P_{(\Gamma,\psi)},P_{(\Delta,\phi)}$ are defined. Then the following are equivalent:
    \begin{enumerate}
        \item Then the corresponding GBS groups are isomorphic.
        \item The limit angles coincide {\rm(}i.e. $P_{(\Gamma,\psi)}=P_{(\Delta,\phi)}${\rm)}.
        \item The vectors defined by the edges of $\Delta$ in the affine representation {\rm(}based at the origin{\rm)} lie inside the limit angle $P_{(\Gamma, \psi)}$.
    \end{enumerate}
\end{introthm}

\begin{ex}
    Let us consider the GBS graph $(\Gamma,\psi)$ of \Cref{fig:introlimitangles}. In this example, one can compute the limit directions, which are given by $\bl^-=(2,5)$ and $\bl^+=(5,2)$, see \Cref{fig:introlimitangles2}.
    
    Consider the GBS graph $(\Delta_1,\phi_1)$, with two loops labeled by $3^3, 2^{9}3^{8}$ and $2^3, 2^{17}3^{7}$ respectively. It is easy to check that $(\Delta_1,\phi_1)$ has the same linear invariants and assignment map as $(\Gamma,\psi)$. Moreover, the two edges of $(\Delta_1,\phi_1)$ define the two vectors $(9,5)$ and $(14,7)$, which fall inside the limit angle $P_{(\Gamma,\psi)}$, as $\frac{2}{5}<\frac{9}{5}< \frac{14}{7}< \frac{5}{2}$. Therefore, we deduce from Theorem \ref{introthm:limit-angles} that $(\Delta_1,\phi_1)$ is isomorphic to $(\Gamma,\psi)$.

    Consider now the GBS graph $(\Delta_2,\phi_2)$, with two loops labeled by $3^3, 2^{6}3^{4}$ and $2^3,2^{28}3^{3}$ respectively. Once again $(\Delta_2,\phi_2)$ has the same linear invariants and assignment map as $(\Gamma,\psi)$. In this case, the two edges of $(\Delta_2,\phi_2)$ define the two vectors $(6,1)$ and $(25,3)$, which fall outside the limit angle $P_{(\Gamma,\psi)}$, as $\frac{2}{5}< \frac{5}{2}<\frac{6}{1}< \frac{25}{3}$. Therefore, from Theorem \ref{introthm:limit-angles} we conclude that $(\Delta_2,\phi_2)$ is not isomorphic to $(\Gamma,\psi)$.

    In fact, among groups with the same linear invariants and assignment map, there are infinitely many possible limit angles, each of them corresponding to an isomorphism class of GBS groups, see \Cref{fig:introlimitangles3}. We refer the reader to \Cref{ex7} for further details.
\end{ex}

\begin{figure}[H]
\centering
\includegraphics[width=\textwidth]{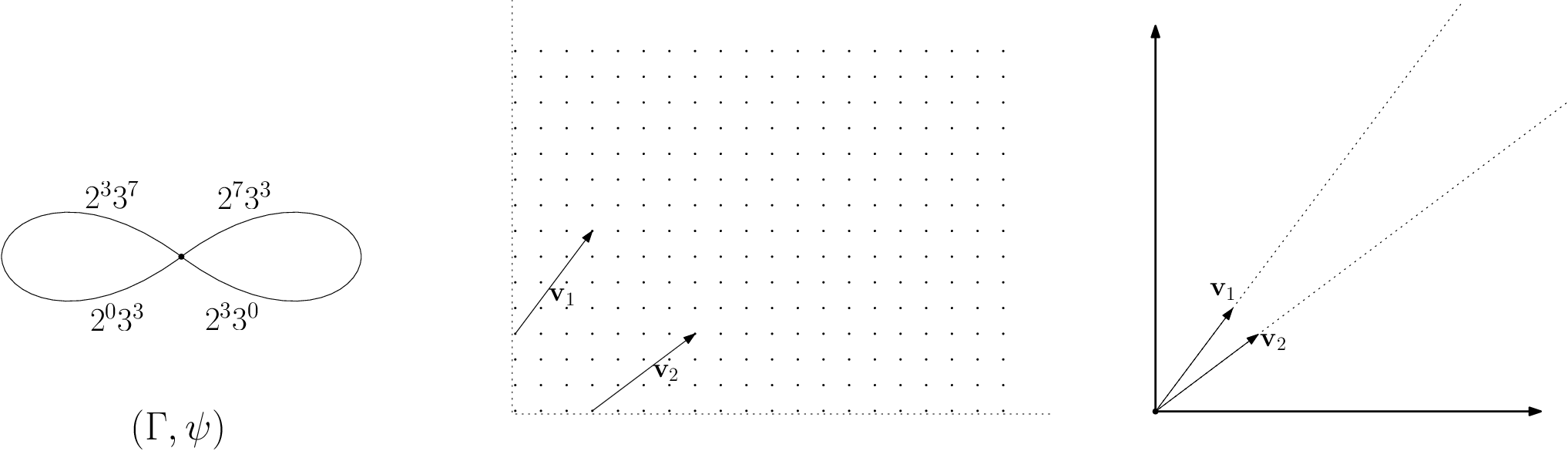}
\centering
\caption{From left to right: the GBS graph $(\Gamma,\psi)$, its affine representation, the two vectors at the origin of the Euclidean plane $\bbR^2$. Note that the picture on the right is not supposed to represent a GBS (but it is useful to emphasize the angle defined by the two vectors).}
\label{fig:introlimitangles}
\end{figure}

\begin{figure}[H]
\centering
\includegraphics[width=0.4\textwidth]{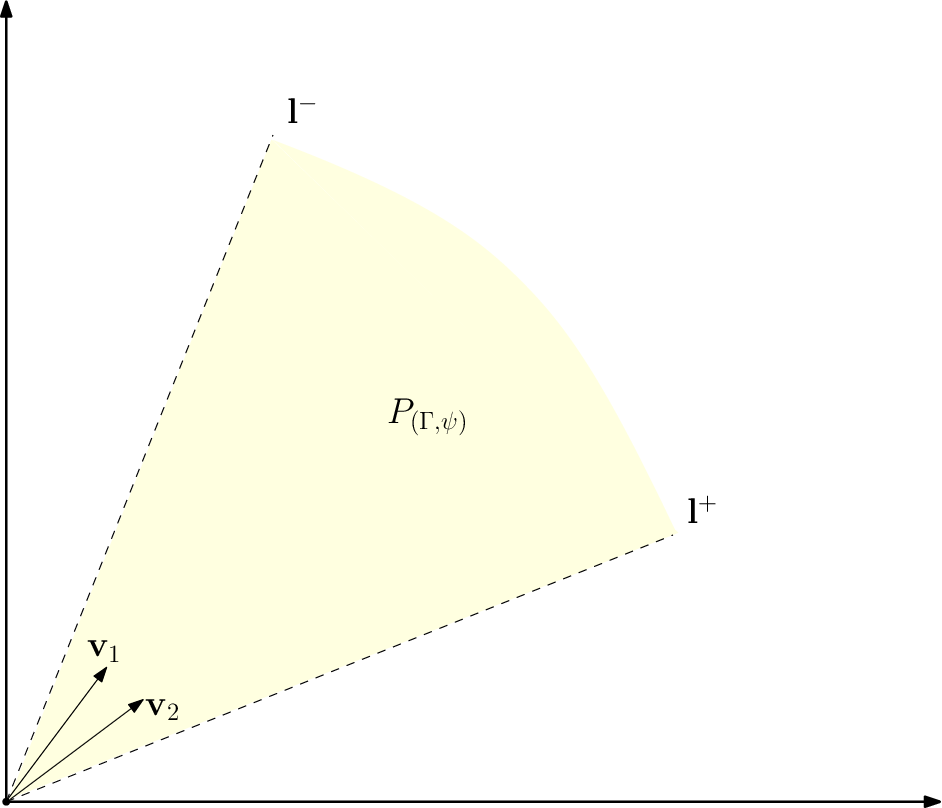}
\centering
\caption{The limit directions $\bl^-,\bl^+$ and the positive cone $P_{(\Gamma,\psi)}$. The two vectors $\bv_1,\bv_2$ will always fall inside the positive cone of their own GBS graph. In order to check isomorphism with another GBS graph $(\Delta,\phi)$, we must computed the vectors associated with $(\Delta,\phi)$, and check whether they fall inside the positive cone $P_{(\Gamma,\psi)}$.}
\label{fig:introlimitangles2}
\end{figure}

A natural question then arises: how many distinct limit angles can occur among GBS groups that have the same linear invariants? In general, this set is infinite, that is, the limit angle invariant can distinguish infinitely many isomorphism classes of groups with the same linear algebra invariants and assignment map. However, the structure of the limit angles given a fixed set of linear invariants is organized as a finite union of arithmetic progressions. In \Cref{introthm:limit-directions-discrete}, we provide a complete description of the set of all possible limit angles. 

\begin{introthm}[\Cref{thm:limit-directions-discrete}]\label{introthm:limit-directions-discrete}
    Let us fix the set of linear invariants and assignment map of a non-degenerate {\rm(}non-exceptional{\rm)} GBS graph with one vertex and two edges. Then, 
    \begin{enumerate}
        \item The set of directions that can be realized as limit directions is a finite union of arithmetic progressions.
        \item This finite set of arithmetic progressions can be computed algorithmically. 
        \item If $P$ is the open cone of all vectors with strictly positive components, then the set of limit directions has no accumulation point in the interior of $P$ {\rm(}i.e. it is finite inside any closed sub-cone of $P${\rm)}.
    \end{enumerate}
\end{introthm}

\begin{figure}[H]
\centering
\includegraphics[width=0.6\textwidth]{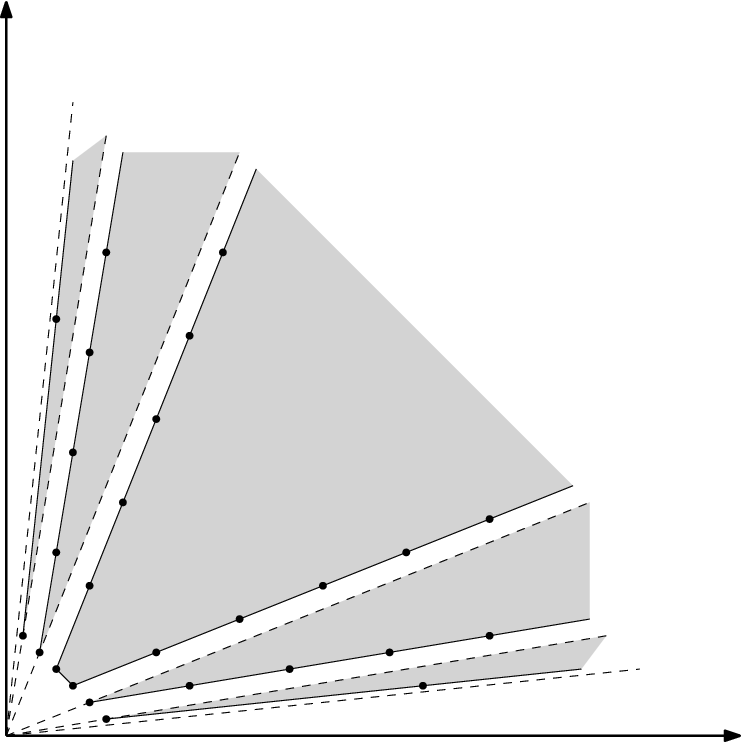}
\centering
\caption{Fixed the linear invariants and the assignment map, there are usually infinitely many possible limit angles. In the picture, we represent finitely many of them: each gray region corresponds to an isomorphism class of GBS groups.}
\label{fig:introlimitangles3}
\end{figure}

\subsection*{Acknowledgements}

This work was supported by the Basque Government grant IT1483-22. The second author was supported by the Spanish Government grant PID2020-117281GB-I00, partly by the European Regional Development Fund (ERDF), the MICIU /AEI /10.13039/501100011033 / UE grant PCI2024-155053-2.

\section{GBS graphs and quasi-conjugacy classes}

In this section we set up the notation, and we review some notions from \cite{ACK-iso1}. 

\subsection{Graphs of groups}

We consider graphs as combinatorial objects, following the notation of \cite{Ser77}. A \textbf{graph} is a quadruple $\Gamma=(V,E,\ol{\cdot},\iota)$ consisting of a set $V=V(\Gamma)$ of \textit{vertices}, a set $E=E(\Gamma)$ of \textit{edges}, a map $\ol{\cdot}:E\rar E$ called \textit{reverse} and a map $\iota:E\rar V$ called \textit{initial endpoint}; we require that, for every edge $e\in E$, we have $\ol{e}\not=e$ and $\ol{\ol{e}}=e$. For an edge $e\in E$, we denote with $\tau(e)=\iota(\ol{e})$ the \textit{terminal endpoint} of $e$. A \textbf{path} in a graph $\Gamma$, with \textit{initial endpoint} $v\in V(\Gamma)$ and \textit{terminal endpoint} $v'\in V(\Gamma)$, is a sequence $\sigma=(e_1,\dots,e_\ell)$ of edges $e_1,\dots,e_\ell\in E(\Gamma)$ for some integer $\ell\ge0$, with the conditions $\iota(e_1)=v$ and $\tau(e_\ell)=v'$ and $\tau(e_i)=\iota(e_{i+1})$ for $i=1,\dots,\ell-1$. A graph is \textbf{connected} if for every pair of vertices, there is a path going from one to the other. For a connected graph $\Gamma$, we define its \textbf{rank} $\rank{\Gamma}\in\bbN\cup\{+\infty\}$ as the rank of its fundamental group (which is a free group).

\begin{defn}
A \textbf{graph of groups} is a quadruple
$$\cG=(\Gamma,\{G_v\}_{v\in V(\Gamma)},\{G_e\}_{e\in E(\Gamma)},\{\psi_e\}_{e\in E(\Gamma)})$$
consisting of a connected graph $\Gamma$, a group $G_v$ for each vertex $v\in V(\Gamma)$, a group $G_e$ for every edge $e\in E(\Gamma)$ with the condition $G_e=G_{\ol{e}}$, and an injective homomorphism $\psi_e:G_e\rar G_{\tau(e)}$ for every edge $e\in E(\Gamma)$.
\end{defn}

Let $\cG=(\Gamma,\{G_v\}_{v\in V(\Gamma)},\{G_e\}_{e\in E(\Gamma)},\{\psi_e\}_{e\in E(\Gamma)})$ be a graph of groups. Define the \textbf{universal group} $\FG{\cG}$ as the quotient of the free product $(*_{v\in V(\Gamma)}G_v)*F(E(\Gamma))$ by the relations
\begin{equation*}\label{FGrelations}
\ol{e}=e^{-1} \qquad\qquad \psi_{\ol{e}}(g)\cdot e=e\cdot\psi_e(g)
\end{equation*}
for $e\in E(\Gamma)$ and $g\in G_e$.

Define the \textbf{fundamental group} $\pi_1(\cG,\ot{v})$ of a graph of group $\cG$ with basepoint $\ot{v}\in V(\Gamma)$ to be the subgroup of $\FG{\cG}$ of the elements that can be represented by words such that, when going along the word, we read a path in the graph $\Gamma$ from $\ot v$ to $\ot v$. The fundamental group $\pi_1(\cG,\ot v)$ does not depend on the chosen basepoint $\ot v$, up to isomorphism. 

\subsection{Generalized Baumslag-Solitar groups}

\begin{defn}
A \textbf{GBS graph of groups} is a finite graph of groups
$$\cG=(\Gamma,\{G_v\}_{v\in V(\Gamma)},\{G_e\}_{e\in E(\Gamma)},\{\psi_e\}_{e\in E(\Gamma)})$$
such that each vertex group and each edge group is $\bbZ$.
\end{defn}

A \textbf{Generalized Baumslag-Solitar group} is a group $G$ isomorphic to the fundamental group of some GBS graph of groups.

\begin{defn}
A \textbf{GBS graph} is a pair $(\Gamma,\psi)$ where $\Gamma$ is a finite graph and $\psi:E(\Gamma)\rar\bbZ\setminus\{0\}$ is a function.
\end{defn}

Given a GBS graph of groups $\cG=(\Gamma,\{G_v\}_{v\in V(\Gamma)},\{G_e\}_{e\in E(\Gamma)},\{\psi_e\}_{e\in E(\Gamma)})$, the map $\psi_e:G_e\rar G_{\tau(e)}$ is an injective homomorphism $\psi_e:\bbZ\rar\bbZ$, and thus coincides with multiplication by a unique non-zero integer $\psi(e)\in\bbZ\setminus\{0\}$. We define the GBS graph associated to $\cG$ as $(\Gamma,\psi)$ associating to each edge $e$ the factor $\psi(e)$ characterizing the homomorphism $\psi_e$.
Giving a GBS graph of groups is equivalent to giving the corresponding GBS graph. In fact, the numbers on the edges are sufficient to reconstruct the injective homomorphisms and thus the graph of groups.

\subsection{Reduced affine representation of a GBS graph}

\begin{defn}\label{def:set-of-primes}
    For a GBS graph $(\Gamma,\psi)$, define its \textbf{set of primes}
    $$\cP(\Gamma,\psi):=\{r\in\bbN \text{ prime } : r\divides \psi(e) \text{ for some } e\in E(\Gamma)\}$$
\end{defn}

Given a GBS graph $(\Gamma,\psi)$, consider the finitely generated abelian group
$$\bA:=\bbZ/2\bbZ\oplus\bigoplus\limits_{r\in\cP(\Gamma,\psi)}\bbZ.$$
We denote with $\mathbf{0}\in\bA$ the neutral element. For an element $\ba=(a_0,a_r : r\in\cP(\Gamma,\psi))\in\bA$ (with $a_0\in\bbZ/2\bbZ$ and $a_r\in\bbZ$ for $r\in\cP(\Gamma,\psi)$), we denote $\ba\ge\mathbf{0}$ if $a_r\ge 0$ for all $r\in\cP(\Gamma,\psi)$; notice that we are not requiring any condition on $a_0$. We define the positive cone $\pA:=\{\ba\in\bA : \ba\ge\mathbf{0}\}$.

\begin{defn}\label{def:affine-representation}
Let $(\Gamma,\psi)$ be a GBS graph. Define its \textbf{{\rm(}reduced{\rm)} affine representation} to be the graph $\Lambda=\Lambda(\Gamma,\psi)$ given by:
\begin{enumerate}
\item $V(\Lambda)=V(\Gamma)\times\pA$ is the disjoint union of copies of $\pA$, one for each vertex of $\Gamma$.
\item $E(\Lambda)=E(\Gamma)$ is the same set of edges as $\Gamma$, and with the same reverse map.
\item For an edge $e\in E(\Lambda)$ we write the unique factorization $\psi(e)=(-1)^{a_0}\prod_{r\in\cP(\Gamma,\psi)}r^{a_r}$ and we define the terminal endpoint $\tau_\Lambda(e)=(\tau_\Gamma(e),(a_0,a_r,\dots))$.
\end{enumerate}
For a vertex $v\in V(\Gamma)$ we denote $\pA_v:=\{v\}\times\pA$ the corresponding copy of $\pA$.
\end{defn}

If $\Lambda$ contains an edge going from $p$ to $q$, then we denote $p\edge q$. If $\Lambda$ contains edges from $p_i$ to $q_i$ for $i=1,\dots,m$, then we denote
$$
\begin{cases}
p_1\edge q_1\\
\cdots\\
p_m\edge q_m
\end{cases}
$$
This does not mean that $p_1\edge q_1,\dots,p_m\edge q_m$ are all the edges of $\Lambda$, but only that in a certain situation we are focusing on those edges. If we are focusing on a specific copy $\pA_v$ of $\pA$, for some $v\in V(\Gamma)$, and we have edges $(v,\ba_i)\edge (v,\bb_i)$ for $i=1,\dots,m$, then we say that $\pA_v$ contains edges
$$
\begin{cases}
\ba_1\edge \bb_1\\
\cdots\\
\ba_m\edge \bb_m
\end{cases}
$$
omitting the vertex $v$.

\subsection{Support and control of vectors}

The following notions for elements of $\bA$ will be widely used along the paper.

\begin{defn}\label{def:support}
For $\bx\in\bA$ define its \textbf{support} as the set
$$\supp{\bx}:=\{r\in\cP(\Gamma,\psi) : x_r\not=0\}$$
\end{defn}
\begin{rmk}
Note that we omit the $\bbZ/2\bbZ$ component from the definition of support.
\end{rmk}

\begin{defn}\label{def:control-a}
Let $\ba,\bb,\bw\in\pA$. We say that $\ba,\bw$ \textbf{controls} $\bb$ if any of the following equivalent conditions holds:
\begin{enumerate}
\item We have $\ba\le \bb\le \ba+k\bw$ for some $k\in\bbN$.
\item We have $\bb-\ba\ge\mathbf{0}$ and $\supp{\bb-\ba}\subseteq\supp{\bw}$.
\end{enumerate}
\end{defn}

\subsection{Affine paths and conjugacy classes}\label{sec:conjugacy-classes}

Let $(\Gamma,\psi)$ be a GBS graph and let $\Lambda$ be its affine representation. Given a vertex $p=(v,\ba)\in V(\Lambda)$ and an element $\bw\in\pA$, we define the vertex $p+\bw:=(v,\ba+\bw)\in V(\Lambda)$. For two vertices $p,p'\in V(\Lambda)$ we denote $p'\ge p$ if $p'=p+\bw$ for some $\bw\in\pA$; in particular this implies that both $p,p'$ belong to the same $\pA_v$ for some $v\in V(\Gamma)$.

\begin{defn}
An \textbf{affine path} in $\Lambda$, with \textit{initial endpoint} $p\in V(\Lambda)$ and \textit{terminal endpoint} $p'\in V(\Lambda)$, is a sequence $(e_1,\dots,e_\ell)$ of edges $e_1,\dots,e_\ell\in E(\Lambda)$ for some $\ell\ge0$, such that there exist $\bw_1,\dots,\bw_\ell\in\pA$ satisfying the conditions $\iota(e_1)+\bw_1=p$ and $\tau(e_\ell)+\bw_\ell=p'$ and $\tau(e_i)+\bw_i=\iota(e_{i+1})+\bw_{i+1}$ for $i=1,\dots,\ell-1$.
\end{defn}

The elements $\bw_1,\dots,\bw_\ell$ are called \textit{translation coefficients} of the path; if they exist, then they are uniquely determined by the path and by the endpoints, and they can be computed algorithmically. They mean that an edge $e\in E(\Lambda)$ connecting $p$ to $q$ allows us also to travel from $p+\bw$ to $q+\bw$ for every $\bw\in\pA$.

\begin{defn}\label{def:conjugacy}
    Let $p,q\in V(\Lambda)$.
    \begin{enumerate}
        \item We denote $p\cnj q$, and we say that $p,q$ are \textbf{conjugate}, if there is an affine path going from $p$ to $q$.
        \item We denote $p\lecnj q$ if $p\le q'$ for some $q'\cnj q$.
        \item We denote $p\qcnj q$, and we say that $p,q$ are \textbf{quasi-conjugate}, if $p\lecnj q$ and $q\lecnj p$.
    \end{enumerate}
\end{defn}

The relation $\cnj$ is an equivalence relations on the set $V(\Lambda)$. The relation $\lecnj$ is a pre-order on $V(\Lambda)$, and $\qcnj$ is the equivalence relation induced by the pre-order. Note that if $p\cnj p'$ then $p+\bw\cnj p'+\bw$ for all $\bw\in\pA$. Similarly, if $p\qcnj p'$ then $p+\bw\qcnj p'+\bw$ for all $\bw\in\pA$.

\subsection{Moves on GBS graphs}\label{sec:moves}

We introduce some moves that can be performed on a GBS graph. Each of them induces an isomorphism at the level of fundamental group of the corresponding graph of groups, see \cite{ACK-iso1}. Let $(\Gamma,\psi)$ be a GBS graph.

\paragraph{Vertex sign-change.}

Let $v\in V(\Gamma)$ be a vertex. Define the map $\psi':E(\Gamma)\rar\bZ\setminus\{0\}$ such that $\psi'(e)=-\psi(e)$ if $\tau(e)=v$ and $\psi'(e)=\psi(e)$ otherwise. We say that the GBS graph $(\Gamma,\psi')$ is obtained from $(\Gamma,\psi)$ by means of a \textbf{vertex sign-change}. If $\cG$ is the GBS graph of groups associated to $(\Gamma,\psi)$, then the vertex sign change move corresponds to changing the chosen generator for the vertex group $G_v$.

\paragraph{Edge sign-change.}

Let $(\Gamma,\psi)$ be a GBS graph. Let $d\in E(\Gamma)$ be an edge. Define the map $\psi':E(\Gamma)\rar\bZ\setminus\{0\}$ such that $\psi'(e)=-\psi(e)$ if $e=d,\ol{d}$ and $\psi'(e)=\psi(e)$ otherwise. We say that the GBS graph $(\Gamma,\psi')$ is obtained from $(\Gamma,\psi)$ by means of an \textbf{edge sign-change}. If $\cG$ is the GBS graph of groups associated to $(\Gamma,\psi)$, then the vertex sign change move corresponds to changing the chosen generator for the edge group $G_d$.

\paragraph{Slide.}

Let $d,e$ be distinct edges with $\tau(d)=\iota(e)=u$ and $\tau(e)=v$; suppose that $\psi(\ol{e})=n$ and $\psi(e)=m$ and $\psi(d)=\ell n$ for some $n,m,\ell\in\bbZ\setminus\{0\}$ (see Figure \ref{fig:slide}). Define the graph $\Gamma'$ by replacing the edge $d$ with an edge $d'$; we set $\iota(d')=\iota(d)$ and $\tau(d')=v$; we set $\psi(\ol{d'})=\psi(\ol{d})$ and $\psi(d')=\ell m$. We say that the GBS graph $(\Gamma',\psi)$ is obtained from $(\Gamma,\psi)$ by means of a \textbf{slide}. At the level of the affine representation, we have an edge $p\edge q$ and we have another edge with an endpoint at $p+\ba$ for some $\ba\in\pA$. The slide has the effect of moving the endpoint from $p+\ba$ to $q+\ba$ (see Figure \ref{fig:slide}).

$$\begin{cases}
p\edge  q\\
r\edge  p+\ba
\end{cases}
\xrightarrow{\text{slide}}\quad
\begin{cases}
p\edge  q\\
r\edge  q+\ba
\end{cases}$$

\begin{figure}[H]
\centering

\includegraphics[width=0.75\textwidth]{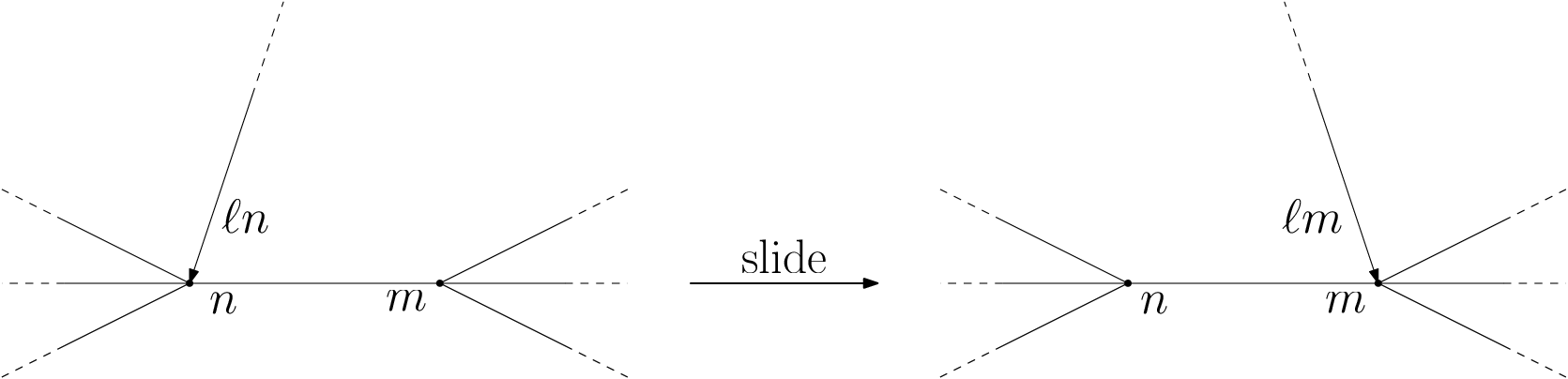}

\vspace{1cm}

\includegraphics[width=\textwidth]{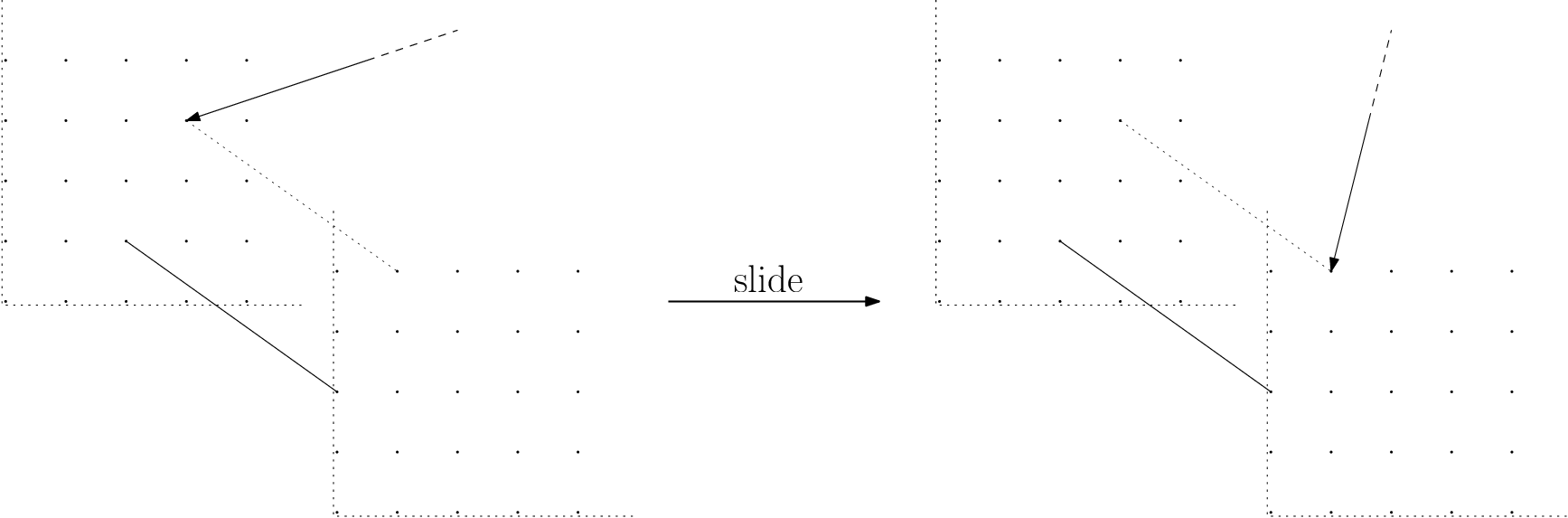}

\caption{An example of a slide move. Above you can see the GBS graphs. Below you can see the corresponding affine representations.}
\label{fig:slide}
\end{figure}

\paragraph{Induction.}

Let $(\Gamma,\psi)$ be a GBS graph. Let $e$ be an edge with $\iota(e)=\tau(e)=v$; suppose that $\psi(\ol{e})=1$ and $\psi(e)=n$ for some $n\in\bbZ\setminus\{0\}$, and choose $\ell\in\bbZ\setminus\{0\}$ and $k\in\bbN$ such that $\ell\divides n^k$. Define the map $\psi'$ equal to $\psi$ except on the edges $d\not=e,\ol{e}$ with $\tau(e)=v$, where we set $\psi'(d)=\ell\cdot\psi(d)$. We say that the GBS graph $(\Gamma,\psi')$ is obtained from $(\Gamma,\psi)$ by means of an \textbf{induction}. At the level of the affine representation, we have an edge $(v,\mathbf{0})\edge (v,\bw)$. We choose $\bw_1\in\pA$ such that $\bw_1\le k\bw$ for some $k\in\bbN$; we take all the endpoints of other edges that are in $\pA_v$, and translate them up by adding $\bw_1$ (see Figure \ref{fig:ind}).

\begin{figure}[H]
\centering

\includegraphics[width=0.65\textwidth]{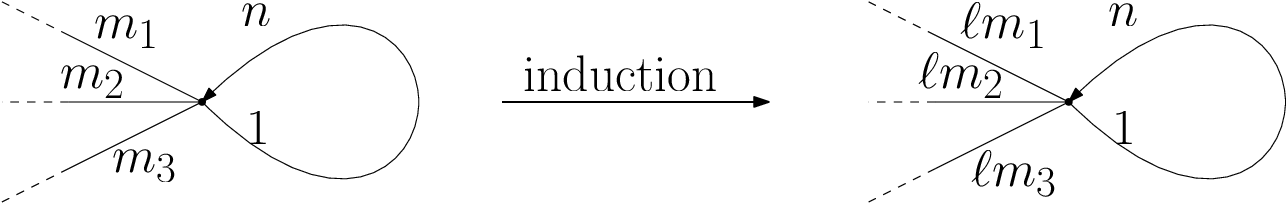}

\vspace{1cm}

\includegraphics[width=0.8\textwidth]{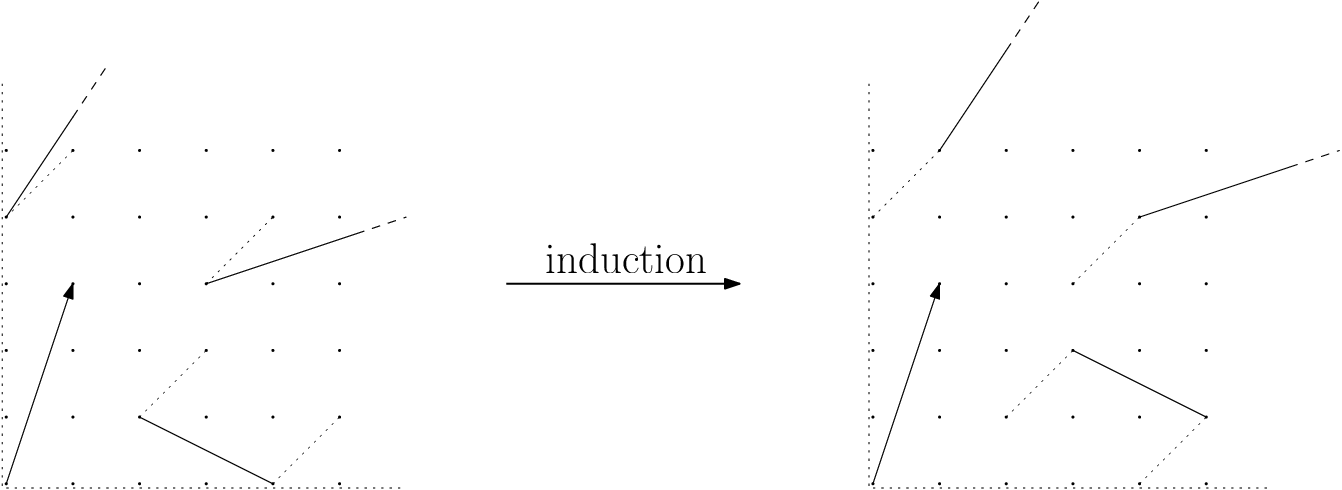}

\caption{An example of an induction move. Above you can see the GBS graphs; here $\ell\divides n^k$ for some integer $k\ge0$. Below you can see the corresponding affine representations.}
\label{fig:ind}
\end{figure}

\paragraph{Swap.}

Let $(\Gamma,\psi)$ be a GBS graph. Let $e_1,e_2$ be distinct edges with $\iota(e_1)=\tau(e_1)=\iota(e_2)=\tau(e_2)=v$; suppose that $\psi(\ol{e}_1)=n$ and $\psi(e_1)=\ell_1 n$ and $\psi(\ol{e}_2)=m$ and $\psi(e_2)=\ell_2 m$ and $n\divides m$ and $m\divides \ell_1^{k_1} n$ and $m\divides \ell_2^{k_2} n$ for some $n,m,\ell_1,\ell_2\in\bbZ\setminus\{0\}$ and $k_1,k_2\in\bbN$ (see Figure \ref{fig:swap}). Define the graph $\Gamma'$ by substituting the edges $e_1,e_2$ with two edges $e_1',e_2'$; we set $\iota(e_1')=\tau(e_1')=\iota(e_2')=\tau(e_2')=v$; we set $\psi(\ol{e_1'})=m$ and $\psi(e_1')=\ell_1 m$ and $\psi(\ol{e_2'})=n$ and  $\psi(e_2')=\ell_2 n$. We say that the GBS graph $(\Gamma',\psi)$ is obtained from $(\Gamma,\psi)$ by means of a \textbf{swap move}. Let $\bw_1,\bw_2\in\pA$ and $p,q\in V(\Lambda)$ be such that $p\le q\le p+k_1\bw_1$ and $p\le q\le p+k_2\bw_2$ for some $k_1,k_2\in\bbN$. At the level of the affine representation, we have an edge $e_1$ going from $p$ to $p+\bw_1$ and an edge $e_2$ going from $q$ to $q+\bw_2$. The swap has the effect of substituting them with $e_1'$ from $q$ to $q+\bw_1$ and with $e_2'$ from $p$ to $p+\bw_2$ (see Figure \ref{fig:swap}).

$$\begin{cases}
p\edge  p+\bw_1\\
q\edge  q+\bw_2
\end{cases}
\xrightarrow{\text{swap}}\quad
\begin{cases}
q\edge  q+\bw_1\\
p\edge  p+\bw_2
\end{cases}$$

\begin{figure}[H]
\centering

\includegraphics[width=0.65\textwidth]{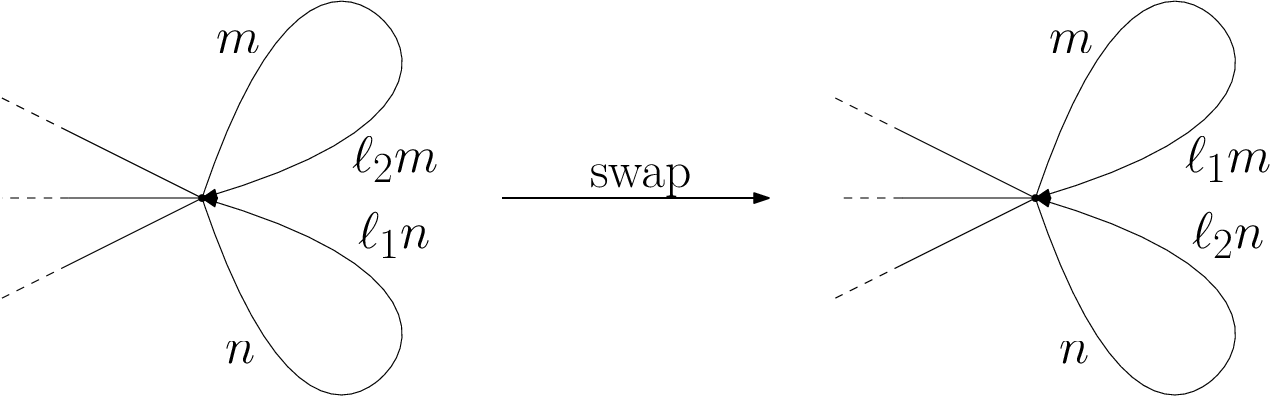}

\vspace{1cm}

\includegraphics[width=0.8\textwidth]{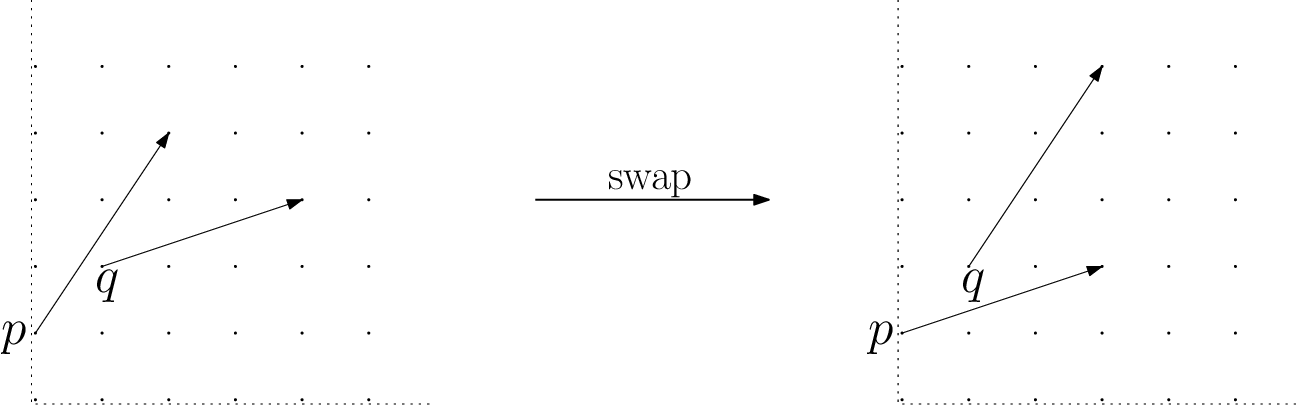}

\caption{An example of a swap move. Above you can see the GBS graphs; here $n\divides m\divides\ell_1^{k_1}n$ and $n\divides m\divides\ell_2^{k_2}n$ for some integers $k_1,k_2\ge0$. Below you can see the corresponding affine representations.}
\label{fig:swap}
\end{figure}

\paragraph{Connection.}\label{sec:connection}

Let $(\Gamma,\psi)$ be a GBS graph. Let $d,e$ be distinct edges with $\iota(d)=u$ and $\tau(d)=\iota(e)=\tau(e)=v$; suppose that  $\psi(\ol{d})=m$ and $\psi(d)=\ell_1 n$ and $\psi(\ol{e})=n$ and $\psi(e)=\ell n$ and $\ell_1\ell_2=\ell^k$ for some $m,n,\ell_1,\ell_2,\ell\in\bbZ\setminus\{0\}$ and $k\in\bbN$ (see Figure \ref{fig:connection}). Define the graph $\Gamma'$ by substituting the edges $d,e$ with two edges $d',e'$; we set $\iota(d')=v$ and $\tau(d')=\iota(e')=\tau(e')=u$; we set  $\psi(\ol{d'})=n$ and $\psi(d')=\ell_2 m$ and $\psi(\ol{e'})=m$ and $\psi(e')=\ell m$. We say that the GBS graph $(\Gamma',\psi)$ is obtained from $(\Gamma,\psi)$ by means of a \textbf{connection move}. Let $\bw,\bw_1,\bw_2\in\pA$ and $k\in\bbN$ be such that $\bw_1+\bw_2=k\cdot\bw$. At the level of the affine representation, we have a two edges $q\edge  p+\bw_1$ and $p\edge  p+\bw$. The connection move has the effect of replacing them with two edges $p\edge  p+\bw_2$ and $q\edge  q+\bw$ (see Figure \ref{fig:connection}).

$$\begin{cases}
q\edge  p+\bw_1\\
p\edge  p+\bw
\end{cases}
\xrightarrow{\text{connection}}\quad
\begin{cases}
p\edge  q+\bw_2\\
q\edge  q+\bw
\end{cases}$$

\begin{rmk}
In the definition of connection move, we also allow for the two vertices $u,v$ to coincide.
\end{rmk}

\begin{figure}[H]
\centering

\includegraphics[width=0.8\textwidth]{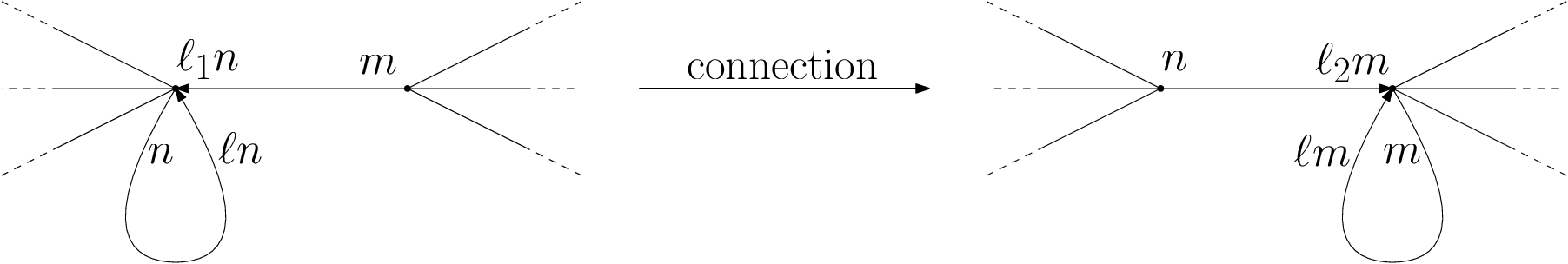}

\vspace{1cm}

\includegraphics[width=\textwidth]{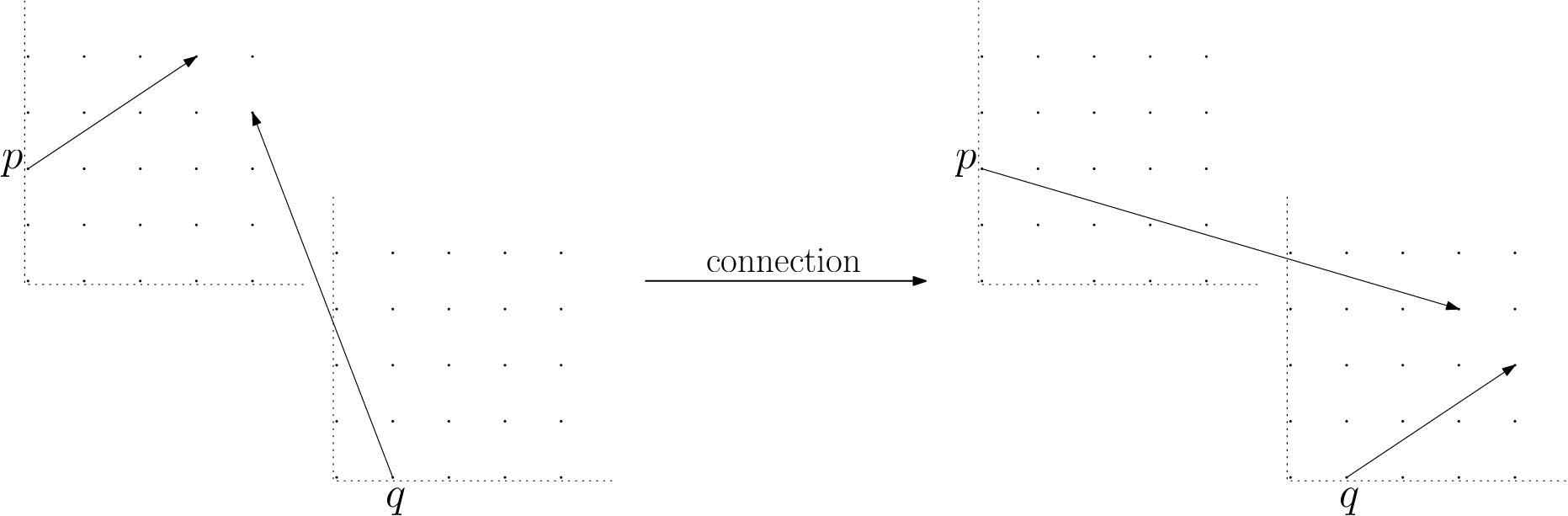}

\caption{An example of a connection move. Above you can see the GBS graphs; here $\ell_1\ell_2=\ell^k$ for some integer $k\ge0$. Below you can see the corresponding affine representations.}
\label{fig:connection}
\end{figure}

\subsection{Sequences of moves}

The following \Cref{thm:sequence-new-moves} is one of the main results from \cite{ACK-iso1}, telling us that, in order to deal with the isomorphism problem for GBSs, it suffices to deal with sequences of moves as described in \Cref{sec:moves} above. For the notion of \textit{totally reduced} GBS graph, we refer the reader to \cite{ACK-iso1}. We point out that the requirement of being totally reduced is not restrictive at all: every GBS graph can be algorithmically changed into a totally reduced one by means of a sequence of moves.

\begin{thm}\label{thm:sequence-new-moves}
Let $(\Gamma,\psi),(\Delta,\phi)$ be totally reduced GBS graphs and suppose that the corresponding graphs of groups have isomorphic fundamental group. Then $\abs{V(\Gamma)}=\abs{V(\Delta)}$ and there is a sequence of slides, swaps, connections, sign-changes and inductions going from $(\Delta,\phi)$ to $(\Gamma,\psi)$. Moreover, all the sign-changes and inductions can be performed at the beginning of the sequence.
\end{thm}
\begin{proof}
    See \cite{ACK-iso1}.
\end{proof}

Thus it suffices to deal with the following problem: given two GBS graphs $(\Gamma,\psi),(\Delta,\phi)$, determine whether there is a sequence of edge sign-changes, inductions, slides, swaps, connections going from one to the other. Note that a sign-change (resp. a slide, an induction, a swap, a connection) induces a natural bijection between the set of vertices of the graph before and after the move. Of course we can ignore the issue of guessing the bijection among the sets of vertices and the sign-changes at the beginning of the sequence, as these choices can be done only in finitely many ways. In what follows, we will also ignore the question of guessing the inductions at the beginning of the sequence, as this usually represents a marginal issue.

\section{GBS graphs with two edges}\label{sec:isomorphism-2edges}

The aim of this section is to solve the isomorphism problem for GBS groups associated to GBS graphs with one vertex and two edges.

The difficult case arises when dealing with configurations as defined in \Cref{def:configuration}. In this setting, slide and connection moves modify two of the four endpoints of the edges, while the other two remain fixed at the minimal points (swaps are almost never be applied). In \Cref{sec:binary-trees}, we show that using only slide moves, one can obtain (a subset of) a rooted binary tree of configurations, with a distinguished configuration serving as the \textit{root} of the tree. In \Cref{sec:binary-trees-connection}, we demonstrate that the only relevant connection moves are those applied at the root of the binary tree. In \Cref{sec:sequence-of-roots}, we describe the set of all possible roots obtainable from a given one through connection moves. This set turns out to be a union of two arithmetic progressions, along with finitely many exceptional configurations. This structure allows us to prove \Cref{thm:main}.

\subsection{Configurations}\label{sec:configurations}

Let $(\Gamma,\psi)$ be a GBS graph with one vertex and two edges. Let $\Lambda$ be its affine representation, which consists of a unique copy of $\pA$, corresponding to the unique vertex of $\Gamma$, and two edges.

\begin{defn}\label{def:configuration}
    A \textbf{configuration} is given by two edges inside $\pA$
    \begin{equation*}
        (C)=\begin{cases}
            \ba_1\edge\ba_1+\bx_1\\
            \ba_2\edge\ba_2+\bx_2
        \end{cases}
    \end{equation*}
    for some $\ba_1,\ba_2\in\pA$ and $\bx_1,\bx_2\in\bA$ satisfying the following properties:
    \begin{enumerate}
        \item $\ba_1\not\ge\ba_2$ and $\ba_2\not\ge\ba_1$.
        \item $\ba_1+\bx_1$ is greater than or equal to at least one of $\ba_1$ or $\ba_2$.
        \item $\ba_2+\bx_2$ is greater than or equal to at least one of $\ba_1$ or $\ba_2$.
    \end{enumerate}
\end{defn}

We say that the configuration $(C)$ has \textit{minimal points} $\ba_1,\ba_2$ and \textit{vectors} $\bx_1,\bx_2$. When the minimal points $\ba_1,\ba_2$ are clear from the context, we will omit then and just say that $(C)$ has vectors $\bx_1,\bx_2$. Note that, if we apply a sequence of moves to a configuration, we obtain another configuration. Moreover, the two minimal points $\ba_1,\ba_2$ will stay the same (except possibly for the $\bbZ/2\bbZ$ component), and thus are an isomorphism invariant of $(\Gamma,\psi)$.

\begin{defn}
    Let $(C)$ be a configuration with minimal points $\ba_1,\ba_2$ and vectors $\bx_1,\bx_2$. We say that $(C)$ is \textbf{degenerate} if it falls into at least one of the following cases:
    \begin{enumerate}
        \item At least one of the vectors $\bx_1,\bx_2$ lies in $\bbZ/2\bbZ\le\bA$.
        \item The two edges are $\ba_1\edge\ba_2+\be$ and $\ba_2\edge\ba_1+\be'$ for some $\be,\be'\in\bbZ/2\bbZ\le\bA$.
        \item $\ba_1+\bx_1\not\ge\ba_2$ and $\ba_2+\bx_2\not\ge\ba_1$.
    \end{enumerate}
\end{defn}

A configuration is degenerate if and only if, applying sequences of slides, swaps, and connections involving only the two edges of the configuration, we can only reach finitely many GBS graphs. In this case, in all the GBS graphs that can be reached, the two edges form a degenerate configuration.

In what follows, we assume that our configurations are non-degenerate. In this case, the only moves that can be applied are slides and connections, and applying these moves, we obtain other (non-degenerate) configurations. We point out that almost all slide moves change one of the two endpoints $\ba_1+\bx_1,\ba_2+\bx_2$, while the endpoints $\ba_1,\ba_2$ are usually not involved. The only exception is the case of twin roots (\Cref{def:twin-roots} below), which will be treated separately.

\subsection{Sons, roots and binary trees}\label{sec:binary-trees}

We are interested in studying the set of all configurations that can be reached from a given one using only slide moves. Suppose that we are given a configuration $(C)$ with minimal points $\ba_1,\ba_2$ and vectors $\bx_1,\bx_2$. If $\ba_1+\bx_1\ge\ba_2$, then we can perform a slide move to obtain the configuration
\begin{equation*}
    (C1)=\begin{cases}
        \ba_1\edge\ba_1+\bx_1+\bx_2\\
        \ba_2\edge\ba_2+\bx_2
    \end{cases}
\end{equation*}
and if $\ba_2+\bx_2\ge\ba_1$ then we can perform a slide move to obtain the configuration
\begin{equation*}
    (C2)=\begin{cases}
        \ba_1\edge\ba_1+\bx_1\\
        \ba_2\edge\ba_2+\bx_2+\bx_1
    \end{cases}
\end{equation*}
We say that the configurations $(C1)$ and $(C2)$ are the \textbf{sons} of the configuration $(C)$. The first son $(C1)$ is the one obtained by changing the first edge and has vectors $\bx_1+\bx_2,\bx_2$; the second son $(C2)$ is the one obtained by changing the second edge and has vectors $\bx_1,\bx_2+\bx_1$ (see \Cref{fig:sons}). Every non-degenerate configuration has at least one son, and at most two.

Denote with $\{1,2\}^*$ the set of all finite words in the letters $1,2$. Then we can consider the configurations $(Cs)$ for $s\in\{1,2\}^*$ obtained by taking iterated sons. We say that the configuration $(Cs)$ \textit{exists} if the sequence of slide moves given by $s$ can be performed, and we say that the configuration $(Cs)$ \textit{does not exist} otherwise. For example, if the configuration $(C1222)$ exists, then it has vectors $\bx_1+\bx_2,\bx_2+3(\bx_1+\bx_2)$.

\begin{figure}[H]
\centering
\includegraphics[width=0.7\textwidth]{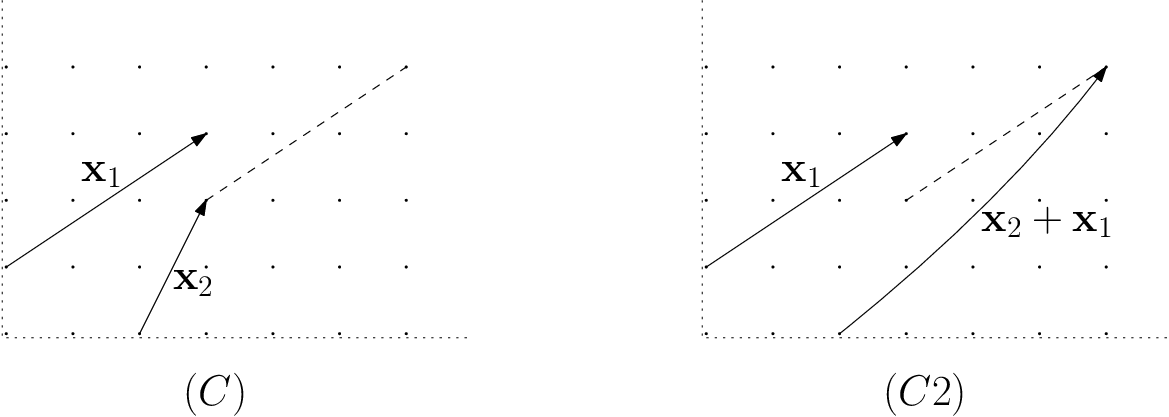}
\centering
\caption{An example of a configuration $(C)$ and of its second son $(C2)$.}
\label{fig:sons}
\end{figure}

\begin{defn}\label{def:root-configuration}
    A configuration $(R)$ with minimal points $\ba_1,\ba_2$ and vectors $\bx_1,\bx_2$ is called \textbf{root} if $\ba_1+\bx_1\not\ge\ba_2+\bx_2$ and $\ba_2+\bx_2\not\ge\ba_1+\bx_1$.
\end{defn}

A configuration is root if and only if it is not the son of any other configuration. Every configuration has a root, and the root is unique in most cases, except in the case of twin roots, as we now explain.

\begin{defn}[Twin roots]\label{def:twin-roots}
    Two non-degenerate roots $(R),(Q)$ are called \textbf{twin roots} if they are of the form
    \[
        (R)=\begin{cases}
            \ba_1\edge\ba_2+\be\\
            \ba_2\edge\ba_2+\bx_2
        \end{cases}
        \qquad\text{and}\qquad
        (Q)=\begin{cases}
            \ba_1\edge\ba_1+\bx_1\\
            \ba_2\edge\ba_1+\be
        \end{cases}
    \]
    with $\be\in\bbZ/2\bbZ\le\bA$ and $\bx_1,\bx_2\not\ge\mathbf{0}$ satisfying $\bx_1+\lambda\ba_1=\bx_2+\lambda\ba_2+\lambda\be$ for some integer $\lambda\ge2$.
\end{defn}

If a root has a twin, then the twin is uniquely determined. It is easy to see that, given two twin roots, we can go from one to the other by means of a sequence of slide moves.

\begin{prop}[Existence and uniqueness of roots]\label{prop:unique-roots}
    Every non-degenerate configuration $(C)$ is an iterated son $(C)=(Rs)$ of a root $(R)$. Moreover, exactly one of the following cases takes place:
    \begin{enumerate}
        \item There is exactly one possible choice of $(R)$ and of $s\in\{1,2\}^*$.
        \item There are exactly two possible choices of $(R)$, and they are twin roots. For each of them, there is a unique possible choice of $s\in\{1,2\}^*$.
    \end{enumerate}
\end{prop}
\begin{proof}
    Take a configuration given by edges $\ba_1\edge\ba_1+\bx_1$ and $\ba_2\edge\ba_2+\bx_2$. This is a first son of some other configuration if and only if $\ba_1+\bx_1\ge\ba_2+\bx_2$, and in that case the ``first father" is uniquely determined. Similarly, it is a second son of another configuration if and only if $\ba_2+\bx_2\ge\ba_1+\bx_1$, and in that case ``second father" is uniquely determined. Thus we can find a unique father (either a first father or a second father) of our configuration; unless $\ba_2+\bx_2=\ba_1+\bx_1+\be$ with $\be\in\bbZ/2\bbZ\in\bA$, in which case we can find exactly one first father and exactly one second father.

    CASE OF TWIN ROOTS: suppose that $\ba_2+\bx_2=\ba_1+\bx_1+\be$ with $\be\in\bbZ/2\bbZ\in\bA$. Then we can go to the first father with edges $\ba_1\edge\ba_2+\be$ and $\ba_2\edge\ba_2+\bx_2$; from here, we are forced to take iterated second fathers with edges $\ba_1\edge\ba_2+\be$ and $\ba_2\edge\ba_2+\bx_2-k(\ba_2-\ba_1+\be)$ for $k\ge0$, and note that at some point this must stop, since $\ba_1-\ba_2\not\ge\mathbf{0}$; thus we have a unique way of reaching a root, which will have edges $\ba_1\edge\ba_2+\be$ and $\ba_2\edge\ba_2+\bx_2+\be-\mu(\ba_2-\ba_1+\be)$ for some $\mu\in\bbN$. Similarly, we can go to the second father with edges $\ba_1\edge\ba_1+\bx_1$ and $\ba_2\edge\ba_1+\be$; from here, we have a unique way of reaching a root, which will have edges $\ba_1\edge\ba_1+\bx_1-\nu(\ba_1-\ba_2+\be)$ and $\ba_2\edge\ba_1+\be$ for some $\nu\in\bbN$. It is easy to check that these are twin roots, and are the unique two roots possible for our configuration.

    If we are not in the ``case of twin roots", then we have a unique father for our configuration, so we can move to this father. Similarly, we keep taking iterated fathers; if at any point we end up in a configuration which has more than one father, then we are in the ``case of twin roots", and we are done. Otherwise, we just have to show that the process terminates, and the conclusion will follow.

    If $\bx_1,\bx_2\ge\mathbf{0}$, then the sum of all of the components of $\bx_1+\bx_2$ strictly decreases every time we take a father (since the configuration is non-degenerate). Thus this can not be the case along a whole infinite sequence. Suppose that at some point we end up, say, in a configuration with $\bx_1\ge\mathbf{0}$ and $\bx_2\not\ge\mathbf{0}$; note that in any case $-\bx_2\not\ge\mathbf{0}$. If $\ba_2+\bx_2\ge\ba_1+\bx_1$, then take iterated second fathers, changing $\bx_2$ into $\bx_2-k\bx_1$ for $k\ge0$, and this must stop since $\bx_1\ge\mathbf{0}$ and the configuration is non-degenerate. If $\ba_1+\bx_1\ge\ba_2+\bx_2$, then we take iterated first fathers, changing $\bx_1$ into $\bx_1-k\bx_2$, and the process must terminate since $-\bx_2\not\ge\mathbf{0}$. The statement follows.
\end{proof}

Fixed a configuration $(C)$, we are interested in describing the family of its iterated sons $(Cs)$ for $s\in\{1,2\}^*$. By Proposition \ref{prop:unique-roots}, we know that the set of iterated sons of $(C)$ is a subset of a rooted binary tree; in fact, if $(Cs)=(Cs')$ then we write $(C)=(Rt)$ for some root $(R)$, and we deduce that $(Rts)=(Rts')$ and thus $ts=ts'$, yielding $s=s'$. Proposition \ref{prop:binary-trees} below gives a precise characterization of which subsets of the binary tree can occur (see Figure \ref{fig:binary-trees}), and in which cases. We are mainly interested in the particular case when $(C)$ is a root but the result is proved without this assumption.

\begin{prop}[Binary trees]\label{prop:binary-trees}
    Suppose that $(C)$ is a non-degenerate configuration with minimal points $\ba_1,\ba_2$ and vectors $\bx_1,\bx_2$. Then exactly one of the following cases happens:
    \begin{enumerate}
        \item\label{itm:2i} $\ba_1+\bx_1\not\ge\ba_2$. In this case, $(Cs)$ exists if and only if $s=2^i$ for some $i\ge0$.
        \item\label{itm:1i} $\ba_2+\bx_2\not\ge\ba_1$. In this case, $(Cs)$ exists if and only if $s=1^i$ for some $i\ge0$.
        \item\label{itm:binary-tree} $\ba_1+\bx_1\ge\ba_2$ and $\ba_2+\bx_2\ge\ba_1$ and 
        $\bx_1,\bx_2\ge\mathbf{0}$. In this case, $(Cs)$ exists for all $s\in\{1,2\}^*$.
        \item\label{itm:truncated1} $\ba_1+\bx_1\ge\ba_2$ and $\ba_2+\bx_2\ge\ba_1$ and 
        $\bx_1\ge\mathbf{0}$ and $\bx_2\not\ge\mathbf{0}$. In this case, let $h\ge0$ be the maximum integer such that $\ba_1+\bx_1+h\bx_2\ge\ba_2$. Then, $(Cs)$ exists if and only if either $s = 1^{h'}2 s'$ for some $0\le h'\le h$ and some $s'\in \{1,2\}^*$, or $s=1^{h+1}2^i$ for some $i\ge0$.
        \item\label{itm:truncated2} $\ba_1+\bx_1\ge\ba_2$ and $\ba_2+\bx_2\ge\ba_1$ and 
        $\bx_1\not\ge\mathbf{0}$ and $\bx_2\ge\mathbf{0}$. In this case, let $k\ge0$ be the maximum integer such that $\ba_2+\bx_2+k\bx_1\ge\ba_1$. Then, $(Cs)$ exists if and only if either $s = 2^{k'}1s'$ for some $0\le k'\le k$ and some $s'\in \{1,2\}^*$, or $s=2^{k+1}1^i$ for some $i\ge0$.
        \item\label{itm:truncated12} $\ba_1+\bx_1\ge\ba_2$ and $\ba_2+\bx_2\ge\ba_1$ and 
        $\bx_1,\bx_2\not\ge\mathbf{0}$. In this case, let $h\ge0$ be the maximum integer such that $\ba_1+\bx_1+h\bx_2\ge\ba_2$, and let $k\ge0$ be the maximum integer such that $\ba_2+\bx_2+k\bx_1\ge\ba_1$. Then $(Cs)$ exists if and only if either $s=1^{h'}2s'$ for some $0\le h'\le h$ and $s'\in\{1,2\}^*$, or $s=2^{k'}1s'$ for some $0\le k'\le k$ and $s'\in\{1,2\}^*$, or $s=1^{h+1}2^i$ for some $i\ge0$, or $s=2^{k+1}1^i$ for some $i\ge0$.
    \end{enumerate}
\end{prop}
\begin{proof}
    \Cref{itm:2i} follows from the fact that the configuration is non-degenerate, and thus $\ba_2+\bx_1\ge\ba_1$. Similarly for \Cref{itm:1i}.
    
    If we are in a configuration satisfying the conditions of \Cref{itm:binary-tree}, then the configuration has two sons, that will still satisfy the conditions of \Cref{itm:binary-tree}. Thus the full binary tree originating from that configuration exists. In particular we obtain \Cref{itm:binary-tree}.

    Suppose that we are in the conditions of \Cref{itm:truncated1}. The configuration $(C2)$ exists, and it satisfies the conditions of \Cref{itm:binary-tree}, and thus the full binary tree over this configuration exists. For every integer $0\le h'\le h$, the configuration $(C1^{h'}2)$ exists, and it satisfies the conditions of \Cref{itm:binary-tree}, and thus the full binary tree over this configuration exists. Finally, the configuration $(C1^{h+1})$ exists, and it satisfies the conditions of \Cref{itm:2i}; the conclusion follows.

    \Cref{itm:truncated2} and \Cref{itm:truncated12} are analogous.
\end{proof}

\begin{figure}[H]
\centering
\includegraphics[width=\textwidth]{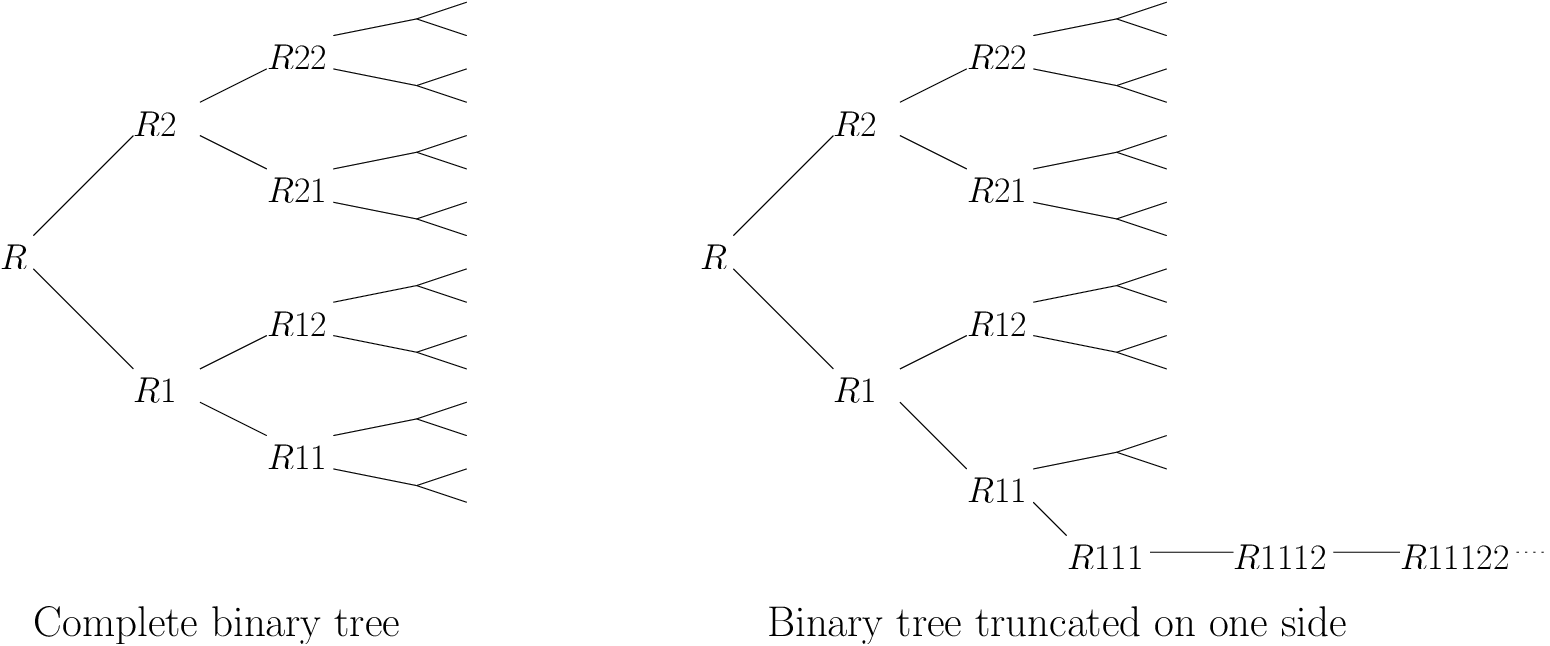}
\centering
\caption{Two examples of subsets of a rooted binary tree. On the left, the tree corresponding to Item \ref{itm:binary-tree} of Proposition \ref{prop:binary-trees}, and on the right, the tree corresponding to Item \ref{itm:truncated1}.}
\label{fig:binary-trees}
\end{figure}

In the particular case of twin roots, the behavior of the binary trees is as follows. Let
\[
    (R)=\begin{cases}
        \ba_1\edge\ba_2+\be\\
        \ba_2\edge\ba_2+\bx_2
    \end{cases}
     \qquad\text{and}\qquad
     (Q)=\begin{cases}
        \ba_1\edge\ba_1+\bx_1\\
         \ba_2\edge\ba_1+\be
    \end{cases}
\]
be non-degenerate twin roots with $\bx_1+\lambda\ba_1=\bx_2+\lambda\ba_2+\lambda\be$ for $\lambda\ge2$ (see \Cref{fig:twin-roots}).

\begin{figure}[H]
\centering
\includegraphics[width=0.9\textwidth]{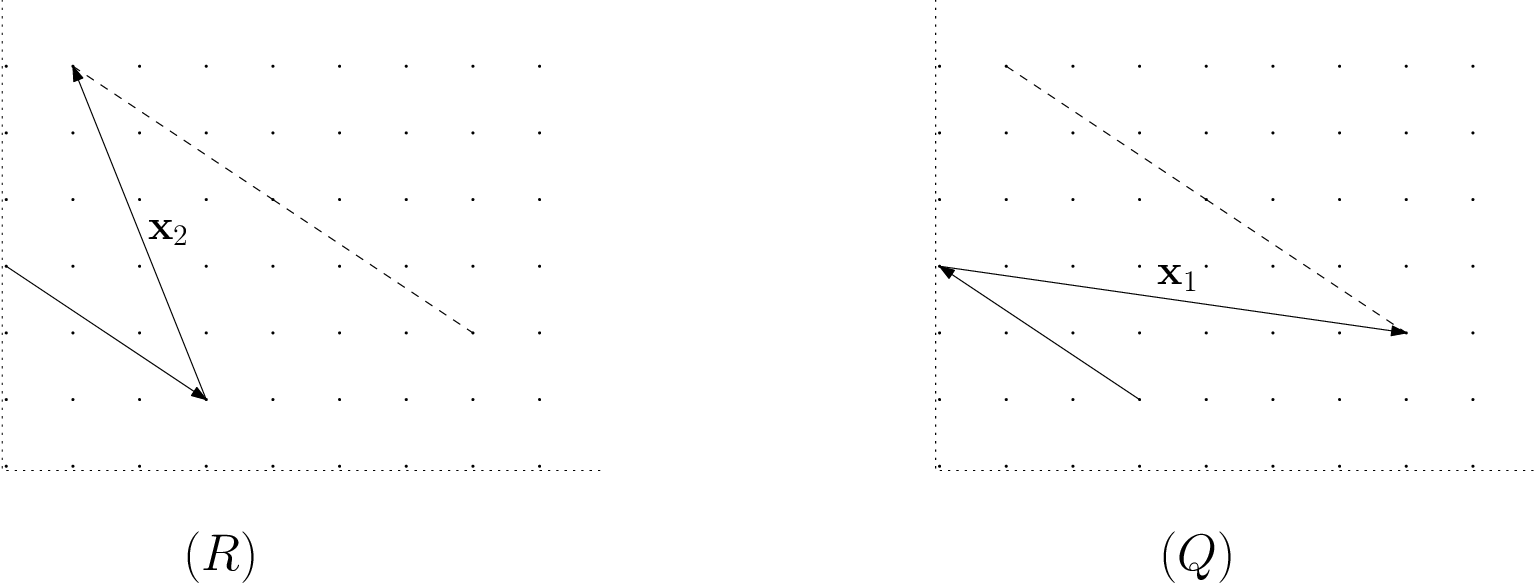}
\centering
\caption{An example of two twin roots. The dashed line witnesses the identity $\bx_1+3\ba_1=\bx_2+3\ba_2$, and hints the sequence of slide moves necessary to pass from one twin to the other.}
\label{fig:twin-roots}
\end{figure}

\begin{lem}[Binary trees for twin roots]\label{lem:twin-binary-trees}
    We have the following:
    \begin{enumerate}
        \item\label{itm:sons-twin-1} For $s\in\{1,2\}^*$, we have that $(Rs)$ exists if and only if either $s$ does not begin with $1,2^{\lambda-1}$, or $s=12^i,2^{\lambda-1}1^i$ for some $i\ge0$.
        \item\label{itm:sons-twin-2} For $s\in\{1,2\}^*$, we have that $(Qs)$ exists if and only if either $s$ does not begin with $2,1^{\lambda-1}$, or $s=21^i,1^{\lambda-1}2^i$ for some $i\ge0$.
        \item\label{itm:sons-twin-common} A configuration $(C)$ is an iterated son of both $(R),(Q)$ if and only if $(C)$ is an iterated son of $(R2^\ell 1)=(Q1^{\lambda-1-\ell}2)$ for some $0\le \ell\le\lambda-1$. In that case, $\ell$ is uniquely determined.
    \end{enumerate}
\end{lem}
\begin{proof}
    \Cref{itm:sons-twin-1} and \Cref{itm:sons-twin-2} follow using \Cref{prop:binary-trees}. For \Cref{itm:sons-twin-common}, we observe that the configurations $(R),(R2),\dots,(R2^{\lambda-1})$ can not be iterated sons of $(Q)$ (since in iterated sons of $(Q)$ the edge $\ba_1\edge\ba_2+\be$ never appears as first edge), and similarly $(Q),(Q1),\dots,(Q1^{\lambda-1})$ can not be iterated sons of $(R)$. It is easy to check that $(R2^\ell 1)=(Q1^{\lambda-1-\ell}2)$ for $0\le \ell\le\lambda-1$. If $(C)$ has more than one root, then by \Cref{prop:unique-roots} we can write $(C)=(Rs)$ for a unique $s\in\{1,2\}^*$, and since $(C)\not=(R),(R2),\dots,(R2^{\lambda-1})$ we obtain that $(s)$ must begin with $2^\ell 1$ for a unique $0\le \ell\le\lambda-1$. The conclusion follows.
\end{proof}

\subsection{Connections in a binary tree}\label{sec:binary-trees-connection}

\Cref{lem:connection-above} tells us that, if we can perform a connection from a configuration $(C)$, then one of the two sons exists, and we can obtain the same result by performing a connection from the son - this means that we can restrict our attention only to connections performed at very deep levels in the trees. On the contrary, \Cref{lem:connection-below} tells us that, if we can apply a connection from a son configuration $(C1)$, then we can obtain the same result by performing a connection from $(C)$ - this means that we can restrict our attention only to connections performed at the roots.

\begin{lem}[Moving to sons]\label{lem:connection-above}
    If a non-degenerate configuration $(C)$ allows for a connection {\rm(}with the first edge controlling the endpoint of the second{\rm)}, then $(C2)$ exists and it allows for a connection {\rm(}with the first edge controlling the endpoint of the second{\rm)}. Moreover, the result of the two connections is the same {\rm(}for suitable choices of the parameter $k$ of {\rm\Cref{sec:connection})}.
\end{lem}
\begin{proof}
    Immediate from the definitions.
\end{proof}

\begin{lem}[Moving to roots]\label{lem:connection-below}
    Let $(C)$ be a non-degenerate configuration and suppose that $(C1)$ exists.
    \begin{enumerate}
        \item $(C1)$ allows for a connection {\rm(}with the first edge controlling the endpoint of the second{\rm)} if and only if $(C2)$ exists. In that case, the result of the connection is $(C2)$ {\rm(}for a suitable choice of the parameter $k$ of {\rm\Cref{sec:connection})}.
        \item If $(C1)$ allows for a connection {\rm(}with the second edge controlling the endpoint of the first{\rm)}, then $(C)$ allows for a connection {\rm(}with the second edge controlling the endpoint of the first{\rm)}. Moreover, the result of the two connections is the same {\rm(}for suitable choices of the parameter $k$ of {\rm\Cref{sec:connection})}.
    \end{enumerate}
\end{lem}
\begin{proof}
    Suppose that $(C)$ has minimal points $\ba_1,\ba_2$ and vectors $\bx_1,\bx_2$. Since $(C1)$ exists, we must have $\ba_1+\bx_1\ge\ba_2$.
    
    If $(C1)$ allows for a connection (with the first edge controlling the endpoint of the second) then $\ba_2+\bx_2\ge\ba_1$ and thus $(C2)$ exists. If $(C2)$ exists then $\ba_2+\bx_2\ge\ba_1$, and thus $\ba_1+\bx_1+\bx_2\ge\ba_2+\bx_2\ge\ba_1$, giving that $(C1)$ allows for a connection (with the first edge controlling the endpoint of the second). It is immediate to check that the result of the connection from $(C1)$ is $(C2)$, as desired.

    The second item is immediate from the definitions.
\end{proof}

\begin{lem}[Jumping between roots]\label{lem:jump-to-root}
    Let $(R)$ be a non-degenerate root and suppose that we are in the hypothesis to perform a connection {\rm(}with the first edge controlling the endpoint of the second{\rm)}. If we perform the connection, choosing the parameter $k$ of {\rm\Cref{sec:connection}} to be the minimum possible integer, then the result of the connection is a non-degenerate root.
\end{lem}
\begin{proof}
    Suppose that $(R)$ has minimal points $\ba_1,\ba_2$ and vectors $\bx_1,\bx_2$. Suppose that $\bx_1\ge\mathbf{0}$ and that $\ba_1,\bx_1$ controls $\ba_2+\bx_2$, so that we are in the hypothesis to perform a connection, with the first edge controlling the endpoint of the second. Consider the minimum integer $k$ such that $\ba_1+k\bx_1\ge\ba_2+\bx_2$, and observe that $k\ge2$, otherwise $(R)$ would not be a root. We now perform a connection obtaining the configuration
    \[
        \begin{cases}
            \ba_1\edge\ba_1+k\bx_1-\bx_2\\
            \ba_2\edge\ba_2+\bx_1
        \end{cases}
    \]
    Suppose that $\ba_1+k\bx_1-\bx_2\ge\ba_2+\bx_1$: then $\ba_1+(k-1)\bx_1\ge\ba_2+\bx_2$ contradicting the minimality of $k$. Suppose that $\ba_1+k\bx_1-\bx_2\le\ba_2+\bx_1$: then $\ba_1+\bx_1\le\ba_1+(k-1)\bx_1\le\ba_2+\bx_2$ contradicting the fact that $(R)$ was a root. This shows that the configuration obtained after the connection is a root, as desired.
\end{proof}

\subsection{The sequence of roots}\label{sec:sequence-of-roots}

\begin{defn}\label{def:previous-and-next-root}
    Let $(R)$ be a non-degenerate root.
    \begin{enumerate}
        \item If $(R)$ allows for a connection {\rm(}with the first edge controlling the endpoint of the second{\rm)}, define the root $(R^-)$ as the one obtained performing the connection, as in {\rm\Cref{lem:jump-to-root}}.
        \item If $(R)$ allows for a connection {\rm(}with the second edge controlling the endpoint of the first{\rm)}, define the root $(R^+)$ as the one obtained performing the connection, as in {\rm\Cref{lem:jump-to-root}}.
    \end{enumerate}
\end{defn}

It is immediate from the definitions that, if $(R^-)$ exists, then $(R^{-+})$ exists and it is equal to $(R)$; and similarly on the other side. Given a non-degenerate root $(R)$, we can consider the sequence of non-degenerate roots
$$\dots,(R^{--}),(R^-),(R),(R^+),(R^{++}),\dots$$
obtained by an iterated application of \Cref{lem:jump-to-root} (see \Cref{fig:binary-trees-connection}). On each side, the sequence can be finite or infinite. If an element $\bv\in\pA$ plays the role of $\bv=\bx_1$ in $(R)$, then we have that the same element plays the role of $\bv=\bx_2$ in $(R^-)$. Thus, associated with the sequence of roots we have a unique sequence
$$\dots,\bv_{-2},\bv_{-1},\bv_0,\bv_1,\bv_2,\dots$$
with $\bv_i\in\bA$, and such that $(R^{+i})$ has vectors $\bv_i,\bv_{i+1}$. As for the sequence of roots, this sequence of elements of $\bA$ can be finite or infinite on each side. Moreover, we have that $\bv_i\ge\mathbf{0}$ except possibly for the first/last element of the sequence, if the sequence is finite on the left/right.

\begin{rmk}
    One might be worried about twin roots appearing in the sequence above. However, it is easy to check that a twin root never allows for any connection. Thus, if we start with a non-twin root, then no root along the sequence is a twin root.
\end{rmk}

\begin{figure}[H]
\centering
\includegraphics[width=\textwidth]{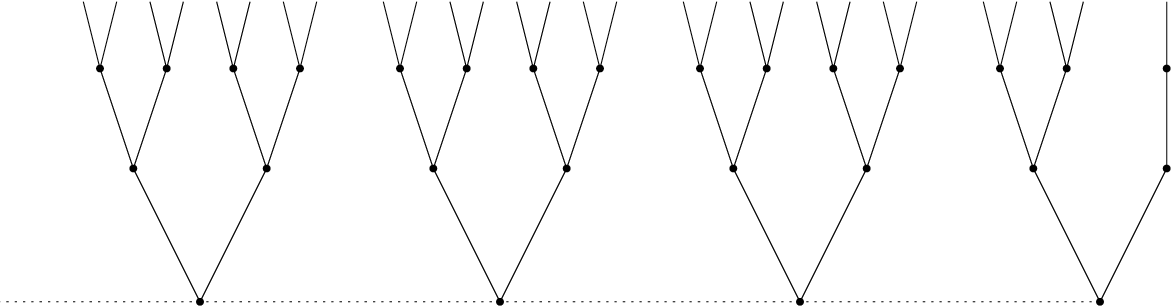}
\centering
\caption{A sequence of binary trees. The roots are related to each other by means of connection moves. On the right, the sequence terminates at a root that allows for connection only on one side. On the left, the sequence is infinite.}
\label{fig:binary-trees-connection}
\end{figure}

The following \Cref{prop:sequence-of-roots} allows us to effectively compute the sequence $\bv_i$ starting from the two initial values $\bv_0,\bv_1$.

\begin{prop}[Sequence of roots]\label{prop:sequence-of-roots}
    Let $(R)$ be a non-degenerate non-twin root, with vectors $\bv_0,\bv_1$. Let $\dots,\bv_{-1},\bv_0,\bv_1,\bv_2,\dots$ be the corresponding sequence of vectors. Then we have the following:
    \begin{enumerate}
        \item\label{itm:compute-next} If $\bv_i,\bv_{i+1},\bv_{i+2}$ exist, then
        \[
            \bv_{i+2}=k\bv_{i+1}-\bv_i,
        \]
        where $k\ge2$ is the minimum natural number such that $k\bv_{i+1}-\bv_i\ge\ba_1-\ba_2$.
        \item\label{itm:compute-previous} If $\bv_{i-1},\bv_i,\bv_{i+1}$ exist, then
        \[
            \bv_{i-1}=h\bv_i-\bv_{i+1},
        \]
        where $h\ge2$ is the minimum natural number such that $h\bv_i-\bv_{i+1}\ge\ba_2-\ba_1$.
        \item\label{itm:compute-last} If the sequence $\dots,\bv_{-2},\bv_{-1},\bv_0,\bv_1,\bv_2,\dots$ is infinite on the right, then $\bv_i\le\bv_{i+1}$ for some $i\in\bbZ$. For every such $i$, the subsequence $\bv_{i+1},\bv_{i+2},\bv_{i+3},\dots$ is an arithmetic progression.
        \item\label{itm:compute-first} If the sequence $\dots,\bv_{-2},\bv_{-1},\bv_0,\bv_1,\bv_2,\dots$ is infinite on the left, then $\bv_{i-1}\ge\bv_i$ for some $i\in\bbZ$. For every such $i$, the subsequence $\dots,\bv_{i-3},\bv_{i-2},\bv_{i-1}$ is an arithmetic progression.
    \end{enumerate}
\end{prop}
\begin{proof}
    \Cref{itm:compute-next} is immediate from the definitions. Note that $k\bv_{i+1}-\bv_i\ge\ba_1-\ba_2$ for $k=0$ would imply that $\ba_2\ge\ba_1+\bv_i$, and thus we would be in the case of a twin root, and no connection would be possible. Note that $k\bv_{i+1}-\bv_i\ge\ba_1-\ba_2$ for $k=1$ would imply that $\ba_2+\bv_{i+1}\ge\ba_1+\bv_i$ contradicting the fact that $(R^{+i})$ is root. Therefore $k\ge2$. Similarly for \Cref{itm:compute-previous}.

    Take $r\in\cP(\Gamma,\psi)$ with the projection $\pi_r:\bA\rightarrow\bbZ$ on the corresponding component, and consider the sequence of integers
    \begin{equation}\label{eq:arithmetic-progressions}
        \dots,\pi_r(\bv_i)-\pi_r(\bv_{i-1}),\pi_r(\bv_{i+1})-\pi_r(\bv_i),\pi_r(\bv_{i+2})-\pi_r(\bv_{i+1}),\dots.
    \end{equation}
    If $\bv_{i-1},\bv_{i+1}$ exist, then $\bv_i\ge\mathbf{0}$ and $\bv_{i-1}+\bv_{i+1}=k\bv_i\ge 2\bv_i$ for some $k\ge3$. Therefore Sequence (\ref{eq:arithmetic-progressions}) is a non-decreasing sequence of integers.

    For \Cref{itm:compute-last}, suppose that the sequence is infinite on the right. For every $r\in\cP(\Gamma,\psi)$, there must be an index $i\in\bbZ$ such that $\pi_r(\bv_{i+1})-\pi_r(\bv_i)\ge0$ (otherwise $\pi_r(\bv_i)\rar-\infty$ for $i\rar+\infty$, contradiction). Since we have finitely many choices of $r\in\cP(\Gamma,\psi)$, there must be an index $i\in\bbZ$ such that $\bv_{i+1}\ge\bv_i$. If $i\in\bbZ$ satisfies $\bv_i\le\bv_{i+1}$, then $\bv_i\le\bv_{i+1}\le\bv_{i+2}\le\dots$ since Sequence (\ref{eq:arithmetic-progressions}) is non-decreasing. For every $j\ge i+1$ we must have $\ba_2+\bv_{j+1}\ge\ba_1$ (since $\bv_{j-1}$ exists), and thus $2\bv_{j+1}-\bv_j\ge\ba_1-\ba_2$, yielding that $\bv_{j+2}=2\bv_{j+1}-\bv_j$. It follows that $\bv_{i+1},\bv_{i+2},\bv_{i+3},\dots$ is an arithmetic progression. \Cref{itm:compute-first} is analogous.
\end{proof}

\subsection{Classification of GBS graphs with one vertex and two edges}

We are now ready to classify GBS graphs of groups with one vertex and two edges.

\begin{thm}\label{thm:main}
    There is an algorithm that, given two totally reduced GBS graphs $(\Gamma,\psi),(\Delta,\phi)$, where $(\Gamma,\psi)$ has one vertex and two edges, decides whether or not there is a sequence of slides, swaps, connections going from $(\Gamma,\psi)$ to $(\Delta,\phi)$. In case such a sequence exists, the algorithm also computes one such sequence.
\end{thm}
\begin{proof}
    We can assume that both GBS graphs have one vertex and two edges. If the two edges of $(\Gamma,\psi)$ are in different quasi-conjugacy classes, then they are not isomorphic as in \cite{ACK-iso2}, it is proven that they quasi-conjugacy class is an isomorphism invariant;
    thus we assume that they belong to a common quasi-conjugacy class. Note that the minimal regions of this quasi-conjugacy class consist basically of a point each (modulo the $\bbZ/2\bbZ$ component).
    
    CASE 1: Suppose that the quasi-conjugacy class has four minimal regions $M_1,M_2,M_3,M_4$. By \cite{ACK-iso2}, 
    the four endpoints of the two edges must lie one in each region. Suppose that the first edge has endpoints in $M_1,M_2$ and the other one, $M_3,M_4$. But then $M_1\cup M_2$ is a quasi-conjugacy class by itself, contradicting the assumption that the two edges are in a quasi-conjugacy class.

    CASE 2: Suppose that the quasi-conjugacy class has three minimal regions $M_1,M_2,M_3$. By \cite{ACK-iso2},
     three of the four endpoints of the edges must lie in the minimal regions, suppose we have one edge with endpoints in $M_1,M_2$ and the other with one endpoint in $M_3$. The other endpoint must lie in one of $M_1,M_2$, otherwise $M_1\cup M_2$ would be a quasi-conjugacy class by itself. It follows that we can reach only finitely many GBS graphs, and so we can describe the isomorphic GBS graphs.

    CASE 3: Suppose that the quasi-conjugacy class has two minimal regions $M_1,M_2$. If one edge has no endpoint in any of them, then by \cite{ACK-iso2} 
    the other endpoint must have one endpoint in each, and thus $M_1\cup M_2$ is a quasi-conjugacy class by itself, a contradiction. Thus each edge must have at least one endpoint in each region, and thus we are in a configuration
    \[
        (R)=\begin{cases}
            \ba_1\edge\ba_1+\bv_1\\
            \ba_2\edge\ba_2+\bv_2
        \end{cases}
    \]
    where $\ba_1,\ba_2$ lie in $M_1,M_2$ respectively. If the configuration is degenerate, then we can only reach finitely many GBS graphs, and we are done. Otherwise, up to slides, we can assume that $(R)$ is a root.
    
    If $(R)$ is not a twin root, then every slide and connection will always lead to another configuration, with the same endpoints $\ba_1,\ba_2$, changing only the two endpoints $\ba_1+\bv_1,\ba_2+\bv_2$. We consider the sequence of vectors $\dots,\bv_{-1},\bv_0,\bv_1,\bv_2,\dots$ as in \Cref{sec:sequence-of-roots}, and note that by \Cref{prop:sequence-of-roots} we can algorithmically compute a finite segment of it, until for each side, the sequence either terminates or becomes an infinite arithmetic progression. Thus $(\Delta,\phi)$ must be a configuration of the same kind, and we can compute a root for $(\Delta,\phi)$, say with edges $\ba_1\edge\ba_1+\bu_1$ and $\ba_2\edge\ba_2+\bu_2$. We algorithmically check whether $\bu_1,\bu_2$ is a subsegment of our finite sequence, or whether it belongs to one the arithmetic progressions. If it does, then we can explicitly compute a sequence of moves going from $(\Gamma,\psi)$ to $(\Delta,\phi)$, otherwise such a sequence of moves does not exist.

    If $(R)$ is a twin root, we call
    \[
        (R)=\begin{cases}
            \ba_1\edge\ba_2+\be\\
            \ba_2\edge\ba_2+\bx_2
        \end{cases}
        \qquad\text{and}\qquad
        (Q)=\begin{cases}
            \ba_1\edge\ba_1+\bx_1\\
            \ba_2\edge\ba_1+\be
        \end{cases}
    \]
    with $\be\in\bbZ/2\bbZ\le\bA$ and $\bx_1,\bx_2\not\ge\mathbf{0}$ satisfying $\bx_1+\lambda\ba_1=\bx_2+\lambda\ba_2+\lambda\be$ for some integer $\lambda\ge2$. If from $(R)$ we perform a slide changing the endpoint $\ba_2$ of the second edge, then we end up with two edges $\ba_1\edge\ba_2+\be$ and $\ba_1+\be\edge\ba_2+\bx_2$, and the root of this configuration is a twin root $(Q')$, with twin $(R')$, given by
    \[
        (R')=\begin{cases}
            \ba_1+\be\edge\ba_2\\
            \ba_2+\be\edge\ba_2+\bx_2+\be
        \end{cases}
        \qquad\text{and}\qquad
        (Q')=\begin{cases}
            \ba_1+\be\edge\ba_1+\bx_1+\be\\
            \ba_2+\be\edge\ba_1
        \end{cases}
    \]
    and note that these coincide with $(R),(Q)$ if $\be$ is trivial. In any case, the configurations that we can reach from $(R)$ by means of moves are exactly the iterated sons of $(R),(Q),(R'),(Q')$. Thus $(\Delta,\phi)$ must be a configuration of the same kind, and we can compute a root for $(\Delta,\phi)$, and check whether it coincides with one of $(R),(Q),(R'),(Q')$. If it does, then we can explicitly compute a sequence of moves going from $(\Gamma,\psi)$ to $(\Delta,\phi)$, otherwise such a sequence of moves does not exist.

    CASE 4: Suppose that the quasi-conjugacy class has one minimal region $M$. Up to performing some slide an swap, we can assume that $(\Gamma,\psi)$ has edges $\ba\edge\ba+\bw$ and $\bb\edge\bb+\bx$, such that $\ba,\bw$ controls $\bb,\bb+\bx$. But now it is a controlled GBS graph, and the conclusion follows from the results of \cite{ACK-iso1}.
\end{proof}

\begin{cor}\label{cor:main}
    There is an algorithm that, given two totally reduced GBS graphs $(\Gamma,\psi),(\Delta,\phi)$, where $(\Gamma,\psi)$ has one vertex and two edges, decides whether or not the corresponding GBS groups are isomorphic, and in case they are, computes a sequence of sign-changes, inductions, slides, swaps, connections going from $(\Gamma,\psi)$ to $(\Delta,\phi)$.
\end{cor}
\begin{proof}
    By \Cref{thm:main}, we only need to deal with the case where we have induction (since sign-changes can be guessed in finitely many ways, see \Cref{thm:sequence-new-moves} in \Cref{sec:moves}). Thus we can assume that $(\Gamma,\psi)$ has edges $\mathbf{0}\edge\bw$ and $\bb\edge\bb+\bx$, and that $(\Delta,\phi)$ has edges $\mathbf{0}\edge\bw'$ and $\bb'\edge\bb'+\bx'$.

    If the two edges lie in the same conjugacy class, then we fall in the cases treated in \cite{ACK-iso1}. Therefore, we can assume that the two edges lie in different quasi-conjugacy classes. In this case, induction will translate the conjugacy class of $\bb\edge\bb+\bx$, and after that only external equivalence is possible. Thus we just have to check whether $\bw=\bw'$ (so that the lower quasi-conjugacy class is dealt with) and $\qcsupp{\bb-\bb'}\subseteq\qcsupp{\bw}$ (so that, up to induction, we can set $\bb=\bb'$) and $\bx-\bx'\in\gen{\bw}$ (so that, after the induction, we can make the two edges equal by means of external equivalence). All of this can be done algorithmically. The statement follows.
\end{proof}

\section{Limit angles}\label{sec:limit-angles}

In this section we give an interpretation of the previous results in terms of angles in the Euclidean plane $\bbR^2$. Each configuration determines an \textit{angle}, and moving away from the root in a binary tree makes the angle narrower. The root of the binary tree, having the largest angle, is able to ``see" more configurations.

A connection move has the effect of rotating the angle: we can choose to rotate to the left or to the right. This allows us to reach new configurations that before we would not see. If we keep performing connections to rotate the angle always in the same direction, we will converge to a direction, which we call \textit{limit direction}. The two limit directions (on the left and on the right) determine an angle, which we call the \textit{limit angle} of the GBS group.

This limit angle is an isomorphism invariant of the GBS group: limit angles for non-isomorphic GBS groups are disjoint. We give a description of the set of all possible limit angles that can appear, showing that this is a discrete set (it can only accumulate at the boundary of the positive cone). See also \Cref{fig:ex7-limit-angles} and \Cref{fig:ex7-limit-angles-2} from the example at the end of the section.

\subsection{Limit directions}


We now give an interpretation of the above results in terms of \textit{angles}. In this section, we give the technical definition of \textit{limit direction} (see \Cref{def:limit-directions} below), which will be crucial in what follows.

\begin{defn}\label{def:full-tree}
    Let $(C)$ be a configuration with minimal points $\ba_1,\ba_2$ and vectors $\bx_1,\bx_2$. We say that $(C)$ is \textbf{full-tree} if $\ba_1+\bx_1,\ba_2+\bx_2\ge\ba_1,\ba_2$.
\end{defn}

This means that $(C)$ is non-degenerate and falls into the case of \Cref{itm:binary-tree} of \Cref{prop:binary-trees}; equivalently, the iterated son $(Cs)$ exists for all $s\in\{1,2\}^*$. Every (iterated) son of a full-tree configuration is again a full-tree configuration.

Let $(R)$ be a non-degenerate non-twin root, with vectors $\bv_0,\bv_1$. Let $\dots,\bv_{-1},\bv_0,\bv_1,\bv_2,\dots$ be the corresponding sequence of vectors, as in \Cref{sec:sequence-of-roots}. We want to characterize all the full-tree configurations that can be obtained from $(R)$ by means of slides and connections. As noted before, every root in the sequence is full-tree (and thus all the iterated sons are full-tree too), except possibly for the first/last root of the sequence (when the sequence is finite on the left/right). If the sequence of roots terminates on the left/right, then the first/last root falls into exactly one of the cases of \Cref{prop:binary-trees}. In each of those cases, it is easy to characterize which of the iterated sons are full-tree.

It is always possible to go from any full-tree configuration to every other full-tree configuration (possibly in a different tree), by means of slides and connections, passing only through full-tree configurations (by \Cref{lem:connection-above}). Thus, in the same way as in \Cref{sec:sequence-of-roots} we defined the sequence of roots, one could define a ``sequence of minimal full-tree configurations" (most of them would be roots, except near the beginning/end of the sequence). However, we are only interested in the first and last elements of the sequence (or in the limits, when the sequence is infinite), as we now explain.

Let $(R)$ be a root with at least one full-tree iterated son. For $s\in\{1,2\}^*$, we say that $(Rs)$ is a \textit{minimal} full-tree iterated son if $(Rs)$ is full-tree, and there is no initial segment $s'\not=s$ such that $(Rs')$ is full-tree. It is easy to check, using \Cref{prop:binary-trees}, that $(R)$ has finitely many minimal full-tree iterated sons $(Rs_1),(Rs_2),\dots,(Rs_k)$ for some integer $k\ge1$ (and these can be algorithmically computed from $(R)$). Among $s_1,\dots,s_k\in\{1,2\}^*$, we take the unique $s_j$ that begins with most digits $2$. We say that $(Rs_j)$ is the \textit{right-most minimal full-tree iterated son} of $(R)$.

\begin{defn}[Limit directions]\label{def:limit-directions}
    Let $(R)$ be a non-degenerate non-twin root, with vectors $\bv_0,\bv_1$. Let $\dots,\bv_{-1},\bv_0,\bv_1,\bv_2,\dots$ be the corresponding sequence of vectors as in {\rm\Cref{sec:sequence-of-roots}}. Define the \textbf{{\rm(}right{\rm)} limit direction} of $(R)$ as the element $\bl^+\in\bA$ obtained as follows:
    \begin{enumerate}
        \item If the sequence is infinite on the right, then we set $\bl^+=\bv_{i+1}-\bv_{i}$ for $i\in\bbN$ big enough.
        \item If the sequence terminates on the right, and some root in the sequence has some full-tree iterated son, then we take the maximum $i\in\bbZ$ such that $(R^{+i})$ has a full-tree iterated son, and we take the right-most minimal full-tree iterated son $(R^{+i}s)$ of $(R^{+i})$. We set $\bl^+$ to be the second vector of $(R^{+i}s)$.

        \item If no root in the sequence has any full-tree iterated son, then we say that $\bl^+$ is \textit{not defined}.

    \end{enumerate}
\end{defn}

Similarly, we can define an order $\preceq_1$ on $\{1,2\}^*$, and the left-most full-tree iterated son of a root $(R)$. In the same way as in \Cref{def:limit-directions}, we can define the \textbf{(left) limit direction} of $(R)$, which will be an element $\bl^-\in\bA$. The two limit directions satisfy several properties:

\begin{itemize}
    \item We have that $\bl^+$ is defined if and only if $\bl^-$ is defined, and if and only if there is at least one full-tree configuration that can be reached from $(R)$ by means of slides and connections.
    \item If $\bl^-,\bl^+$ are defined, then we have that $\bl^-,\bl^+\ge\mathbf{0}$.
    \item If $\bl^-,\bl^+$ are defined, then each of them is part of some basis for $\gen{\bv_0,\bv_1}$. In general, $\bl^-$ and $\bl^+$ together do not generate $\gen{\bv_0,\bv_1}$ - i.e. they are part of two different bases.
\end{itemize}

\begin{rmk}\label{rmk:classification-roots-without-limit-directions}
    It is natural to wonder for which sequences of roots the limit directions do not exist. Such a sequence contains either one or two roots. It can be a single root, falling into the case of \Cref{itm:truncated12} of \Cref{prop:binary-trees} with parameters $h=k=1$. It can consist of a single root, falling into the case of \Cref{itm:1i} or \Cref{itm:2i} of \Cref{prop:binary-trees}. It can consist of two roots, the first falling into the case of \Cref{itm:1i} and the second falling into the case of \Cref{itm:2i} of \Cref{prop:binary-trees}. These are all the possibilities.
\end{rmk}

\subsection{Rank-2 configurations and angles}

\begin{defn}
    A configuration $(C)$ with vectors $\bx_1,\bx_2$ is called \textbf{rank-$2$} if $\gen{\bx_1,\bx_2}\cong\bbZ^2$.
\end{defn}

Let $(R)$ be a non-degenerate root with vectors $\bv_0,\bv_1$, and suppose that $(R)$ is rank-$2$. Then the results about the roots can be reinterpreted in terms of angles in the Euclidean plane $\bbR^2=\gen{\bv_0,\bv_1}\otimes\bbR$. In \Cref{prop:angles}, we show that $(C)$ belongs to the binary tree of $(R)$ if and only if $(C)$ lies inside the \textit{angle} defined by $(R)$ - up to a few exceptional cases (the non-full-tree configurations). However, by means of connection moves, it is possible to escape from this angle; in the subsequent \Cref{thm:limit-angles}, we show that $(C)$ can be obtained from $(R)$ by means of slides and connections if and only if $(C)$ belongs to a wider \textit{limit angle} induced by $(R)$ - again, up to a few exceptional cases (the non-full-tree configurations). In other words, for GBS groups corresponding to rank-$2$ configurations, we have a complete set of invariants classifying them (i.e. minimal points, subgroup generated, limit angle) - once again, up to the few exceptional cases where the limit angles are not defined. In \Cref{prop:limit-angles-disjoint}, we show that distinct limit angles do not intersect, except possibly at their boundary.

\begin{prop}[Angles]\label{prop:angles}
    Let $(C)$ be a rank-$2$ full-tree configuration with vectors $\bx_1,\bx_2$. Then, a configuration $(D)$ with vectors $\by_1,\by_2$ is an iterated son of $(C)$ if and only if the following conditions hold:
    \begin{enumerate}
        \item $\gen{\by_1,\by_2}=\gen{\bx_1,\bx_2}$ and $\by_1,\by_2$ is Nielsen equivalent to $\bx_1,\bx_2$.
        \item\label{itm:angle} We have that $\by_i=\lambda_i\bx_1+\mu_i\bx_2$ for some integers $\lambda_i,\mu_i\ge0$, for $i=1,2$.
    \end{enumerate}
\end{prop}
\begin{rmk}
    \Cref{itm:angle} of the above \Cref{prop:angles} means that the vectors $\by_1,\by_2$ are internal to the positive cone determined by $\bx_1,\bx_2$. We also refer to this positive cone as the \textit{angle} determined by $\bx_1,\bx_2$.
\end{rmk}
\begin{proof}
    If $(D)$ is an iterated son of $(C)$, then the two conditions hold (by induction on the number of iterations of taking the son). Suppose now that the two conditions hold.
    
    Since $\gen{\bx_1,\bx_2}\cong\bbZ^2$, the condition $\gen{\by_1,\by_2}=\gen{\bx_1,\bx_2}$ implies that $\by_1=\lambda\bx_1+\mu\bx_2$ for some unique coprime integers $\lambda,\mu$, and by hypothesis we must have $\lambda,\mu\ge0$. We now start with the configuration $(C)$ with vectors $\bx_1,\bx_2$, and with the numbers $\lambda,\mu$; we take iterated sons to perform the Euclidean algorithm, as follows.

    If $\mu\ge\lambda>0$, then we take the first son $(C1)$ with vectors $\bx_1'=\bx_1+\bx_2,\bx_2'=\bx_2$, and we take the integers $\lambda'=\lambda,\mu'=\mu-\lambda$ such that $\by_1=\lambda'\bx_1'+\mu'\bx_2$. If $\lambda>\mu>0$, then we take the second son $(C2)$ with vectors $\bx_1'=\bx_1,\bx_2'=\bx_2+\bx_1)$, and we take the integers $\lambda'=\lambda-\mu,\mu'=\mu$ such that $\by_1=\lambda'\bx_1'+\mu'\bx_2'$. Note that, in both cases, the pair $\bx_1',\bx_2'$ is Nielsen equivalent to $\bx_1,\bx_2$, and $\lambda',\mu'\ge0$ are coprime. We then reiterate the procedure with $(C1)$ or $(C2)$ respectively; and so on.
    
    The procedure will stop with an iterated son $(Cs)$, for some $s\in\{1,2\}^*$, associated with vectors $(\bu_1,\bu_2)$ and coprime non-negative integers $(\eta,\theta)$, such that (i) the pair $\bu_1,\bu_2$ is Nielsen equivalent to $\bx_1,\bx_2$, (ii) $\by_1=\eta\bu_1+\theta\bu_2$, and (iii) either $\eta=0$ or $\theta=0$.

    CASE 1: Suppose that $\theta=0$. Then $\eta=1$ (since they are coprime) and thus $\by_1=\bu_1$. Since $\bu_1,\bu_2$ is Nielsen equivalent to $\by_1,\by_2$ (and they generate a $\bbZ^2$), we must have that $\by_2=\bu_2+k\bu_1$ for some integer $k\in\bbZ$. If $k\ge0$ then $(D)=(Cs2^k)$ and we are done. Otherwise $(D2^{\abs{k}})=(Cs)$ with $k<0$: we take $\abs{k}$ to be the smallest such that there is such an $s$, and in particular $s$ must terminate with $1$. Say $s=t1$, and we get that $(Ct)$ has vectors $\bu_1-\bu_2,\bu_2$ and $(D2^{\abs{k}-1})$ has vectors $\bu_1,\bu_2-\bu_1$. But then $\bu_1-\bu_2$ and $\bu_2-\bu_1$ are both combinations with non-negative coefficients of $\bx_1,\bx_2\ge\mathbf{0}$, contradiction.

    CASE 2: Suppose that $\eta=0$. Then $\theta=1$ (since they are coprime) and thus $\by_1=\bu_2$. Since $\bu_1,\bu_2$ is Nielsen equivalent to $\by_1,\by_2$, we must have that $\by_2=-\bu_1+k\bu_2$ for some integer $k\in\bbZ$. This means that we can jump from some iterated son of $(C)$ to $(D)$ by means of a connection. By \Cref{lem:connection-below} we obtain that either $(D)$ is an iterated son of $(C)$ (in which case we are done) or $(D)$ is obtained from $(C)$ itself by means of a connection. But if we obtain $(C)$ from $(D)$ by means of a connection, then $\by_2$ can not be written as a linear combination of $\bx_1,\bx_2$ with non-negative coefficients, contradiction.
\end{proof}

Let $(R)$ be a non-degenerate non-twin rank-$2$ root with vectors $\bv_0,\bv_1$. Let $\dots,\bv_{-1},\bv_0,\bv_1,\bv_2,\dots$ be the corresponding sequence of vectors as in \Cref{sec:sequence-of-roots}. We want to characterize all rank-$2$ configurations that can be obtained from $(R)$ by means of slides and connections.

\begin{thm}[Limit angles]\label{thm:limit-angles}
    Let $(R)$ be a non-degenerate non-twin rank-$2$ root with vectors $\bv_0,\bv_1$, and with limit directions $\bl^-,\bl^+$. Then, a full-tree configuration $(D)$ with vectors $\by_1,\by_2$ can be obtained from $(R)$ by means of slides and connections if and only if the following conditions hold:
    \begin{enumerate}
        \item $\gen{\by_1,\by_2}=\gen{\bv_0,\bv_1}$ and $\by_1,\by_2$ is Nielsen equivalent to $\bv_0,\bv_1$.
        \item\label{itm:limit-angle} We have that $\nu_i\by_i=\lambda_i\bl^-+\mu_i\bl^+$ for some integers $\lambda_i,\mu_i\ge0$ and $\nu_i>0$, for $i=1,2$.
    \end{enumerate}
\end{thm}

\begin{rmk}
    \Cref{itm:limit-angle} of the above \Cref{thm:limit-angles} means that the vectors $\by_1,\by_2$ are internal to the (rational) positive cone induced by $\bl^-,\bl^+$. We also refer to this positive cone as the \textit{limit angle} determined by $\bl^-,\bl^+$.
\end{rmk}

\begin{proof}
    OBSERVATION 1: Let $\bp,\bq,\br,\bs\ge\mathbf{0}$ be non-zero. Suppose that $\bq$ is a combination with strictly positive rational coefficients of $\bp,\br$, and $\br$ is a combination with strictly positive rational coefficients of $\bq,\bs$. Then $\bq,\br$ are combinations with strictly positive rational coefficients of $\bp,\bs$.

    In fact, write $\bq=\lambda\bp+\mu\br$ and $\br=\eta\bq+\theta\bs$ with $\lambda,\mu,\eta,\theta>0$. We deduce that $(1-\mu\eta)\bq=\lambda\bp+\mu\theta\bs$; since $\bp,\bs\ge\mathbf{0}$ are non-zero, we deduce that $1-\mu\eta>0$ and the conclusion follows.

    OBSERVATION 2: Let $\bp,\bq,\br,\bx\ge\mathbf{0}$ be non-zero. Suppose that $\bq,\bx$ are combinations with non-negative rational coefficients of $\bp,\br$. Then $\bx$ is a combination with non-negative rational coefficients of either $\bp,\bq$ or $\bq,\br$.

    In fact, write $\bq=\lambda\bp+\mu\br$ and $\bx=\eta\bp+\theta\br$ and observe that $\lambda\bx=\eta\bq+(\lambda\theta-\mu\eta)\br$ and $\mu\bx=\theta\bq+(\mu\eta-\lambda\theta)\bp$. If $\lambda=0$ or $\mu=0$ we are done; otherwise, we use one or the other identity, depending on the sign of $\lambda\theta-\mu\eta$, and we are done.

    STEP 1: We prove that for all full-tree configurations, which are iterated sons of some root in the sequence, has vectors that can be written as combinations of $\bl^-,\bl^+$ with non-negative rational coefficients.
    
    We note that, in the sequence $\dots,\bv_{i-1},\bv_i,\bv_{i+1},\dots$, every term can be written as a combination of the two adjacent ones with strictly positive rational coefficients.

    We note that if the sequence is infinite on the right, then for $j\in\bbZ$ big enough we have that $\bl^+=\bv_j-\bv_{j-1}$ and thus $\bv_j$ can be written as a linear combination of $\bv_{j-1}$ and $\bl^+$ with strictly positive rational coefficients.

    We note that if $\bv_{j-1},\bv_j,\bv_{j+1}$ exist, and if $\bl^+$ is a vector in some iterated son of $\bv_j,\bv_{j+1}$, then we can write $\bl^+=\lambda\bv_j+\mu\bv_{j+1}$ for some integers $\lambda\ge0$ and $\mu>0$, and $\bv_{j-1}+\bv_{j+1}=k\bv_j$ for some integer $k\ge2$. But then $\bv_j=\frac{\mu}{k\mu+\lambda}\bv_{j-1}+\frac{1}{k\mu+1}\bl^+$ is a combination of $\bv_{j-1},\bl^+$ with strictly positive rational coefficients.

    Putting all of this together, and using Observation 1, we deduce that, for every full-tree root in the sequence, its vectors are combinations of $\bl^-,\bl^+$ with non-negative rational coefficients. If it is the case, the first/last tree in the sequence containing full-tree configurations is dealt with by hand. The conclusion follows.

    STEP 2: Suppose that $(D)$ is a full-tree configuration with vectors $\by_1,\by_2$ that satisfies the two conditions. We observe that $\bl^-,\bl^+$ are linearly independent, since they generate $\by_1,\by_2$. For simplicity, we assume that the sequence of vectors is finite on the left and infinite on the right, the other cases being similar.

    Since $\by_1,\by_2$ are linearly independent, one of them can be written as a combination of $\bl^-,\bl^+$ with non-zero coefficient of $\bl^+$. Suppose that $\by_1=\lambda\bl^-+\mu\bl^+$ with $\lambda\ge0$ and $\mu>0$ rationals, the case with $\by_2$ being analogous.

    We have that $\frac{\bv_j}{j}\rar\bl^+$ for $j\rar+\infty$, and thus for all $j$ big enough we can write $\by_1=\lambda_j\bl^-+\mu_j\frac{\bv_j}{j}$, and we must have that $\lambda_j\rar\lambda$ and $\mu_j\rar\mu$. In particular $\by_1$ must be a combination of $\bl^-,\bv_j$ with non-negative rational coefficients for all $j$ big enough.

    By Step 1 we know that $\dots,\bv_{j-1},\bv_j$ are all combinations with non-negative rational coefficients of $\bl^-,\bv_j$ (except possibly for the one/two initial terms of the sequence). By an iterated application of Observation 2, we deduce that $\by_1$ is a combination with non-negative rational coefficients of some full-tree configuration $(C)$ with vectors $\bx_1,\bx_2$, obtained from $(R)$ with slides and connections.

    We write $\by_1=\eta\bx_1+\theta\bx_2$ with $\eta,\theta\ge0$ rationals and $\by_2=\sigma\bx_1+\tau\bx_2$ with $\sigma,\tau$ rationals. We observe that, since $\by_1,\by_2$ is Nielsen equivalent to $\bx_1,\bx_2$, and since they generate a group isomorphic to $\bbZ^2$, we have that $\eta,\theta,\sigma,\tau$ are uniquely determined, and thus they must be integers. If $\eta=0$ then we must have that $\theta=1$ and $\by_1,\by_2$ is related to $\bx_1,\bx_2$ by slides and possibly a connection, and we are done. Similarly, if $\theta=0$ then we are done. If $\eta,\theta>0$, then up to changing $\by_2$ to $\by_2+h\by_1$ by means of slide moves, for $h$ very big, we can assume that $\sigma,\tau>0$. The conclusion follows by \Cref{prop:angles}.
\end{proof}

\begin{prop}[Different limit angles have disjoint interiors]\label{prop:limit-angles-disjoint}
    Let $(R),(S)$ be non-degenerate non-twin rank-$2$ roots with vectors $\bv_0,\bv_1$ and $\bu_0,\bu_1$ respectively, and with limit directions $\bl^+,\bl^-$ and $\bm^+,\bm^-$ respectively. Suppose that $\gen{\bv_0,\bv_1}=\gen{\bu_0,\bu_1}$ and $\bv_0,\bv_1$ is Nielsen equivalent to $\bu_0,\bu_1$. Then the {\rm(}closed, rational{\rm)} positive cones determined by $\bl^+,\bl^-$ and by $\bm^+,\bm^-$ satisfy exactly one of the following possibilities:
    \begin{enumerate}
        \item The cones coincide. In this case $\bl^+=\bm^+$ and $\bl^-=\bm^-$ and $(R),(S)$ belong to the same sequence of roots.
        \item The cones do not coincide, but they intersect non-trivially. In this case either $\bl^+=\bm^-$ {\rm(}and the intersection is exactly the line $\gen{\bl^+}=\gen{\bm^-}${\rm)} or $\bl^-=\bm^+$ {\rm(}and the intersection is exactly the line $\gen{\bl^-}=\gen{\bm^+}${\rm)}.
        \item The cones are disjoint.
    \end{enumerate}
\end{prop}
\begin{proof}
    Suppose that the interiors of the cones intersect. Then we must have that one of $\bm^+,\bm^-$, say $\bm^-$, is a combination of $\bl^+,\bl^-$ with strictly positive rational coefficients.

    Suppose that, by performing slides and connections on $(R')$, we can obtain a full-tree configuration $(C')$ with vectors $\bm^-$ and some other vector. Then by \Cref{thm:limit-angles} we obtain that $(C')$ can be obtained also from $(R)$ by performing slides and connections. Thus in this case $(R)$ and $(S)$ can be obtained from each other by means of connections, and the two cones coincide.

    Suppose that the sequence of vectors $\dots,\bu_{-1},\bu_0,\bu_1,\dots$ is infinite on the left. Then we have that $\frac{\bu_j}{\abs{j}}\rar\bm_-$ for $j\rar-\infty$ and thus for all but finitely many integers $j<0$ we have that $\frac{\bu_j}{\abs{j}}$ is a combination of $\bl^+,\bl^-$ with strictly positive rational coefficients. But then some full-tree root in the sequence of $(S)$ has both vectors $\bu_{j-1},\bu_j$ which are combinations of $\bl^+,\bl^-$ with strictly positive rational coefficients. By \Cref{thm:limit-angles}, this root must also belong to the sequence of roots associated with $(R)$. Therefore $(R)$ and $(S)$ can be obtained from each other by means of connections, and the two cones coincide.
\end{proof}

\subsection{The space of limit angles}

Fix two minimal points $\ba_1,\ba_2\in\pA$. Consider two vectors $\bh_1,\bh_2\in\bA$ generating a subgroup $\gen{\bh_1,\bh_2}\cong\bbZ^2$. We are interested in characterizing all the isomorphism classes of configurations $(C)$ with vectors $\bx_1,\bx_2$ such that $\gen{\bx_1,\bx_2}=\gen{\bh_1,\bh_2}$ and $\bx_1,\bx_2$ is Nielsen equivalent to $\bh_1,\bh_2$.

In order to do this, the main step is characterizing which directions can be realized as limit directions; this is done with the following \Cref{prop:realizing-limit-directions}. As we had already observed, a limit direction $\bl^+$ always satisfies $\bl^+\ge\mathbf{0}$ and is always part of some basis for $\gen{\bh_1,\bh_2}$.

\begin{prop}[Realizing limit directions]\label{prop:realizing-limit-directions}
    Let $\bh_1,\bh_2\in\bA$ be such that $\gen{\bh_1,\bh_2}\cong\bbZ^2$. Let $\bl\in\gen{\bh_1,\bh_2}$ be an element which is part of some basis {\rm(}but not together with $\ba_1-\ba_2${\rm)}, and such that all components of $\bl$ are strictly positive. Then exactly one of the following cases takes place.
    \begin{enumerate}
        \item $\ba_2+\bl\not\ge\ba_1$. In this case, there is a full-tree configuration $(C)$, with vectors Nielsen equivalent to $\bh_1,\bh_2$, with limit angle $\bl^+=\bl$. The sequence of roots is uniquely determined, and is infinite on the right.
        \item\label{itm:seq-fin} $\ba_2+\bl\ge\ba_1$ and $\ba_1+\bl\not\ge\ba_2$. In this case there is a full-tree configuration $(C)$, with vectors Nielsen equivalent to $\bh_1,\bh_2$, and with limit angle $\bl^+=\bl$. The sequence of roots is uniquely determined, and is finite on the right.
        \item $\ba_2+\bl\ge\ba_1$ and $\ba_1+\bl\ge\ba_2$. In this case, there is no full-tree configuration $(C)$, with vectors Nielsen equivalent to $\bh_1,\bh_2$, with limit angle $\bl^+=\bl$.
    \end{enumerate}
\end{prop}
\begin{rmk}
    Analogous conclusions holds with $\bl^-$ instead of $\bl^+$. In particular, if all components of $\bl$ are strictly positive, then $\bl$ can be realized as a right limit direction if and only if $\bl$ can be realized as a left limit direction, and if and only if $\bl\not\ge\abs{\ba_1-\ba_2}$ (the absolute value taken componentwise).
\end{rmk}
\begin{rmk}
    If $\bl$ is part of a basis together with $\ba_1-\ba_2$, then $\bl$ appears in some full-tree configuration which has two twin roots.
\end{rmk}
\begin{rmk}
    In the above \Cref{prop:realizing-limit-directions}, instead of requiring that all components of $\bl$ are strictly positive, it suffices to require it only for the components in $\qcsupp{\ba_1-\ba_2}$. The requirement does not in any case include the component $\bbZ/2\bbZ$.
\end{rmk}
\begin{proof}
    We denote with $\sim$ the relation of Nielsen equivalence.     Take $\bp\in\bA$ such that $\bp,\bl\sim \bh_1,\bh_2$. Note that $\bp$ is uniquely determined up to adding multiples of $\bl$.

    Suppose that $\bl$ is the limit angle of a sequence of roots $\dots,\bv_i,\bv_{i+1},\dots$ infinite on the right, with $\bv_i,\bv_{i+1}$ Nielsen equivalent to $\bh_1,\bh_2$. Then, for $i\in\bbZ$ big enough, we have that $\bv_i,\bl=\bv_i,\bv_{i+1}-\bv_i\sim\bv_i,\bv_{i+1}\sim\bh_1,\bh_2$ and thus $\bv_i=\bp+A\bl$ and $\bv_{i+1}=\bp+(A+1)\bl$ for some integer $A\in\bbZ$. But since $\bv_i,\bv_{i+1}$ is a root, we get that $\ba_2+\bp+(A+1)\bl\not\ge\ba_1+\bp+A\bl$ and thus $\ba_2+\bl\not\ge\ba_1$ as desired.
    
    Note that the sequence of roots is uniquely determined in this case (up to shift).
    
    Conversely, suppose that $\ba_2+\bl\not\ge\ba_1$. Since all components of $\bl$ are strictly positive, it is easy to check that for $A\in\bbN$ big enough the elements $\bv_0=\bp+A\bl$ and $\bv_1=\bp+(A+1)\bl$ define a full-tree root, whose sequence is infinite on the right and has limit angle $\bl^+=\bl$.

    Suppose that $\bl$ is the limit angle of a sequence of roots finite on the right. Then there must be a full-tree configuration $(C)$ with vectors $\bx,\bl$, which is an iterated son of some root of the sequence, and with the following additional property: if we perform a connection (with the second edge controlling the endpoint of the first) we do not get a full-tree configuration. Since $(C)$ is full-tree, we deduce that $\ba_2+\bl\ge\ba_1$ and that we can perform a connection (with the first edge controlling the endpoint of the first). When we perform the connection we get $\bl,-\bx+B\bl$, and, for $B\in\bbN$ big enough, we have that $-\bx+B\bl\ge\mathbf{0}$ and $\ba_2-\bx+B\bl\ge\ba_1$ (since all the components of $\bl$ are strictly positive); if this newly obtained configuration is not full-tree, it must be because $\ba_1+\bl\not\ge\ba_2$, as desired.
    
    Note that the sequence of roots is uniquely determined in this case, since $\bx,\bl\sim\bh_1,\bh_2$ and being full-tree uniquely determines $\bx$ up to slide moves.

    Conversely, suppose that $\ba_2+\bl\ge\ba_1$ and $\ba_1+\bl\not\ge\ba_2$. Since all components of $\bl$ are strictly positive, we can choose $\bp\in\bA$ in such a way that $\bp,\bl$ is a full-tree configuration (and $\bp,\bl\sim\bh_1,\bh_2$). Note that, by performing a connection (with the second edge controlling the endpoint of the first) we obtain the configurations $\bl,-\bp+B\bl$, which will never be full-tree since $\ba_1+\bl\not\ge\ba_2$. Thus the root of this configuration will give the desired sequence.
\end{proof}

In order to get a full-tree configuration, a necessary condition is that the subgroup $\gen{\bh_1,\bh_2}$ contains at least one element with all components $>0$ (or at least, the ones in $\qcsupp{\ba_1-\ba_2}$). In that case, the Euclidean plane $\gen{\bh_1,\bh_2}\otimes\bbR$ will contain a non-empty open cone $P$, generated by linear combinations with positive coefficients of elements with positive components. Note that the complement $P^c$ has non-empty interior. The two lines in the boundary $\partial P$ correspond to directions of $\gen{\bh_1,\bh_2}\otimes\bbR$ where some component is equal to $0$; note that such directions are not necessarily realized in $\gen{\bh_1,\bh_2}$.

\begin{thm}\label{thm:limit-directions-discrete}
    Let $\bh_1,\bh_2\in\bA$ be such that $\gen{\bh_1,\bh_2}\cong\bbZ^2$. Suppose that $\gen{\bh_1,\bh_2}$ contains at least one vector whose components are all strictly positive. Then we have the following:
    \begin{enumerate}
        \item\label{tyuis} The set of directions that can be realized as limit directions is a finite union of arithmetic progressions in $\gen{\bh_1,\bh_2}$.
        \item\label{aho} This finite set of arithmetic progressions can be algorithmically computed from $\ba_1,\ba_2,\bh_1,\bh_2$.
        \item\label{tjs} If $P\subseteq\gen{\bh_1,\bh_2}$ is the open cone defined above, then the set of limit directions has no accumulation point in the interior of $P$ {\rm(}i.e. it is finite inside any closed sub-cone of $P${\rm)}.
    \end{enumerate}
\end{thm}
\begin{proof}
    Let $\bl$ be a limit direction. There are at most two possibilities for $\bl$ with some component equal to zero; thus we assume that all components of $\bl$ are strictly positive. According to \Cref{prop:realizing-limit-directions} at least one of $\bl\not\ge\ba_1-\ba_2$ or $\bl\not\ge\ba_2-\ba_1$ holds, so assume that $\bl\not\ge\ba_1-\ba_2$. The inequality can fail only on finitely many components, so assume that it fails on the $p_0$-th component for some $p_0\in\cP(\Gamma,\psi)$, and let $c$ be the $p_0$-th component of $\ba_1-\ba_2$. We can write $\bl=x\bh_1+y\bh_2$ for some $x,y\in\bbZ$ coprime. Call $a,b$ the $p_0$-th components of $\bh_1,\bh_2$ respectively. Then we must have $0<xa+yb<c$ and thus $xa+yb$ can only assume finitely many values. Assume that $xa+yb=d$ for some fixed $d\in\bbN$.
    
    If the equation $xa+yb=d$ has a solution $x_0a+y_0b=d$, then all the other solutions are given by $(x_0+\lambda\frac{b}{(a,b)})a+(y_0-\lambda\frac{a}{(a,b)})b=d$ for $\lambda\in\bbZ$. We observe that $(x_0+\lambda\frac{b}{(a,b)},y_0-\lambda\frac{a}{(a,b)})$ divides $(ax_0+\lambda\frac{ab}{(a,b)},by_0-\lambda\frac{ba}{(a,b)})=(ax_0+by_0,by_0-\lambda\frac{ba}{(a,b)})$ which divides $ax_0+by_0=d$. In particular,
    \[
        (x_0+\lambda\frac{b}{(a,b)},y_0-\lambda\frac{a}{(a,b)})=
        (x_0+\lambda'\frac{b}{(a,b)},y_0-\lambda'\frac{a}{(a,b)})
    \]
    whenever $\lambda'-\lambda$ is multiple of $d$. Thus we can define the set $\cS=\{(x,y)\in\bbZ^2 : ax+by=d$ and $x,y$ coprime$\}$, and we have that $\cS$ is a finite union of arithmetic progressions. We now intersect the set $\cS$ with the conditions that $x\bh_1+y\bh_2$ has all components $>0$ and we get a finite union of arithmetic progressions (possibly truncated on one or both sides). This describes the set of all possible values for $\bl=x\bh_1+y\bh_2$ as a finite union of arithmetic progressions, yielding \Cref{tyuis}.

    All the above steps can be performed algorithmically, and thus we get \Cref{aho}.

    For \Cref{tjs}, if we have a sequence of limit directions $\bl_n$ converging to some direction, then up to taking a subsequence we can assume that they all belong to a common arithmetic progression, and thus $\bl_n=\bt+k_n\br$ for some $\bt,\br\in\gen{\bh_1,\bh_2}$ and for $k_n\rar+\infty$ integers, and thus the limit of the sequence of directions must be direction of $\br\in\gen{\bh_1,\bh_2}$. Without loss of generality, we can also assume that $\br$ is not a proper power (as we only care about the direction of $\br$).
    
    But if all components of $\br$ are strictly positive, then $\br$ can be realized as limit direction (by \Cref{prop:realizing-limit-directions}), and thus all $\bl_n$ for $n$ big enough must belong to some full-tree configuration with limit direction $\br$, contradiction.
\end{proof}

\subsection{Non-rank-2 configurations}

We finally deal with the case of configurations which are not rank-$2$. Fix two minimal points $\ba_1,\ba_2\in\pA$.

\begin{prop}\label{prop:rank-1-realizing-limit-directions}
    Let $\bh_1,\bh_2\in\bA$ be such that $\gen{\bh_1,\bh_2}\cong\bbZ$ or $\bbZ\oplus\bbZ/2\bbZ$. Let $\bl\in\gen{\bh_1,\bh_2}$ be an element which is part of a generating pair, and such that all components of $\bl$ are strictly positive. Then exactly one of the following cases takes place:
    \begin{enumerate}
    \item $\ba_2+\bl\not\ge\ba_1$. In this case, there is a full-tree configuration $(C)$, with vectors Nielsen equivalent to $\bh_1,\bh_2$, with limit angle $\bl^+=\bl$. Moreover, there are finitely many sequences of roots for such a configuration $(C)$, and each of them is infinite on the right.
    \item $\ba_2+\bl\ge\ba_1$ and $\ba_1+\bl\not\ge\ba_2$.  In this case, there is a full-tree configuration $(C)$, with vectors Nielsen equivalent to $\bh_1,\bh_2$, with limit angle $\bl^+=\bl$. Moreover, there are finitely many sequences of roots for such a configuration $(C)$, and each of them is finite on the right.
    \item $\ba_2+\bl\ge\ba_1$ and $\ba_1+\bl\ge\ba_2$. In this case, there is no full-tree configuration $(C)$ with limit direction $\bl^+=\bl$.
    \end{enumerate}
\end{prop}
\begin{rmk}
    In the above \Cref{prop:rank-1-realizing-limit-directions}, instead of requiring that all components of $\bl$ are strictly positive, it suffices to require it only for the components in $\qcsupp{\ba_1-\ba_2}$. The requirement does not in any case include the component $\bbZ/2\bbZ$.
\end{rmk}
\begin{proof}
    Suppose first that $\gen{\bh_1,\bh_2}\cong\bbZ$ generated by $\bz$. Write $\bl=\ell\bz$ for some integer $\ell$, which we can assume to be $\ge1$ (meaning that $\bz$ has all components strictly positive).

    If $\dots,\bv_i,\bv_{i+1},\dots$ is a sequence of roots infinite on the right, with limit direction $\bl^+=\bl$, then for $i\in\bbN$ big enough we must have that $\bv_i=\lambda\bz$ and $\bv_{i+1}=\lambda\bz+\bl$ for some integer $\lambda\ge1$. Since $\bv_i,\bv_{i+1}$ generate $\gen{\bz}$ we deduce that $(\lambda,\ell)=1$. Since $\bv_i,\bv_{i+1}$ form a root, we must have that $\ba_2+\lambda\bz+\bl\not\ge\ba_1+\lambda\bz$ meaning that $\ba_2+\bl\not\ge\ba_1$, as desired.

    Conversely, suppose that $\ba_2+\bl\not\ge\ba_1$. Then we choose $\lambda\in\bbN$ big enough and coprime with $\ell$, and we define the root $\lambda\bz,\lambda\bz+\bl$. This gives us a sequence of roots, infinite on the right, with limit direction $\bl^+=\bl$. Note also that changing $\lambda$ by adding or subtracting $\ell$ will give the same sequence of roots (shifted by one). Since there are finitely many (and at least one) residues modulo $\ell$ which are coprime with $\ell$, we obtain that there are finitely many sequences of roots, infinite on the right, with limit direction $\bl^+=\bl$.

    Suppose now that $\bl$ is the limit angle of a sequence of roots finite on the right. Let $\lambda\bz,\mu\bz$ be the last root of the sequence, for some integers $\lambda,\mu\ge1$ with $(\lambda,\mu)=1$. This must fall into \Cref{itm:2i} of \Cref{prop:binary-trees}, meaning that $\ba_1+\lambda\bz\not\ge\ba_2$. In order for the sequence to have limit directions, we must also have $\ba_2+\lambda\bz\ge\ba_1$. In this case the limit direction of the sequence of roots is $\ell^+=\lambda\bz$, meaning that $\lambda=\ell$, as desired.

    Conversely, suppose that $\ba_2+\bl\ge\ba_1$ and $\ba_1+\bl\not\ge\ba_2$. Then we choose $\lambda\in\bbN$ big enough and coprime with $\ell$, and we define the configuration $\ell\bz,\lambda\bz$. This gives us a sequence of roots, finite on the right, with limit direction $\bl^+=\ell$. Note also that changing $\lambda$ by adding or subtracting $\ell$ gives us the same sequence of roots. Since there are finitely many (and at least one) residues modulo $\ell$ which are coprime with $\ell$, we obtain that there are finitely many sequences of roots, infinite on the right, with limit direction $\bl^+=\bl$.
    
    In the case where $\gen{\bh_1,\bh_2}\cong\bbZ\oplus\bbZ/2\bbZ$, we consider the element $\bepsilon\in\bA$ of order two, and we write $\bl+\bepsilon=\ell\bz$ with $\ell\ge1$ maximum possible. We must have that $\gen{\bz,\bepsilon}=\gen{\bh_1,\bh_2}$ (otherwise $\bl$ is not part of a basis for $\gen{\bh_1,\bh_2}$), and all components of $\bz$ are strictly positive. Note that $a\bz+b\bepsilon,c\bz+d\bepsilon$ is a basis for $\gen{\bh_1,\bh_2}$ if and only if $a,c$ are coprime and the determinant $ad-bc$ is odd. Using this, the proof proceeds the same as before.
\end{proof}

\begin{cor}\label{cor:rank-1-limit-directions-discrete}
    Let $\bh_1,\bh_2\in\bA$ be such that $\gen{\bh_1,\bh_2}\cong\bbZ$ or $\bbZ\oplus\bbZ/2\bbZ$. Then there are finitely many isomorphism classes of full-tree configurations $(C)$ with vectors Nielsen equivalent to $\bh_1,\bh_2$. Moreover, this finite set of isomorphism classes can be algorithmically computed.
\end{cor}

\subsection{Examples}\label{sec:examples}

For simplicity, we consider only examples $(\Gamma,\psi)$ where all the labels on all the edges are positive. We use $\bA=\bbZ^{\cP(\Gamma,\psi)}$, omitting the $\bbZ/2\bbZ$ summand. We ignore sign-change moves.

\begin{ex}\label{ex7}
Consider the GBS graph $(\Gamma,\psi)$ with one vertex and two edges, as in \Cref{fig:ex7-GBSgraph}, and call $\Lambda$ its affine representation. We have that $\cP(\Gamma,\psi)=\{2,3\}$. The only vertex of $\Gamma$ corresponds to a single copy of $\pA$ in $\Lambda$. The two edges belong to a common quasi-conjugacy class, with two minimal regions, corresponding to the points $\ba_1,\ba_2\in\pA$ with $\ba_1=(0,3)$ and $\ba_2=(3,0)$. The affine representation has edges $(0,3)\edge(11,13)$ and $(3,0)\edge(18,13)$.

\begin{figure}[H]
\centering
\includegraphics[width=\textwidth]{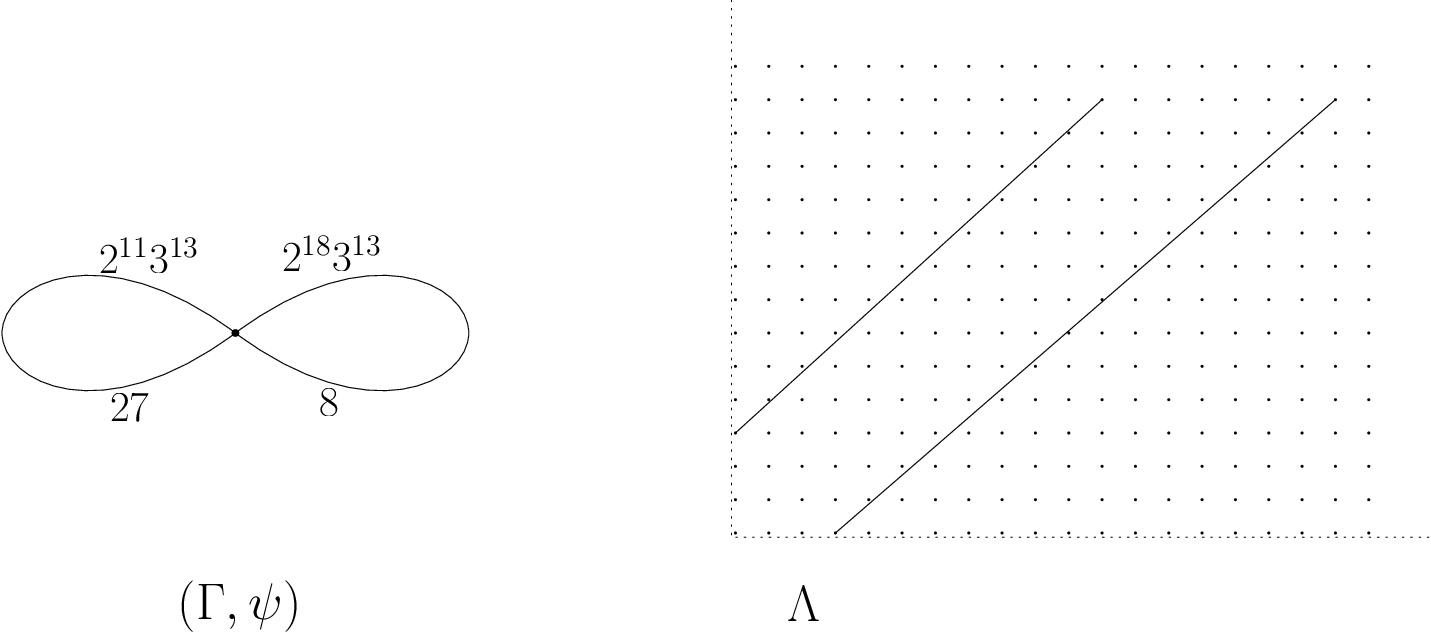}
\centering
\caption{The GBS graph $(\Gamma,\psi)$ of \Cref{ex7} and the corresponding affine representation $\Lambda$.}
\label{fig:ex7-GBSgraph}
\end{figure}

This gives us a configuration $(C)$ with vectors $\bx_1=(11,10),\bx_2=(15,13)$. It is fairly easy to compute the root $(R)$ of this configuration, which has vectors $\bv_1=(3,4),\bv_2=(4,3)$. One readily checks that $(C)=(R112)$, see \Cref{fig:ex7-configuration}.
\[
    (R)=\begin{cases}
        (0,3)\edge(0,3)+(3,4)\\
        (3,0)\edge(3,0)+(4,3)
    \end{cases}
    \qquad
    (C)=(R112)=\begin{cases}
        (0,3)\edge(0,3)+(11,10)\\
        (3,0)\edge(3,0)+(15,13)
    \end{cases}
\]

\begin{figure}[H]
\centering
\includegraphics[width=0.9\textwidth]{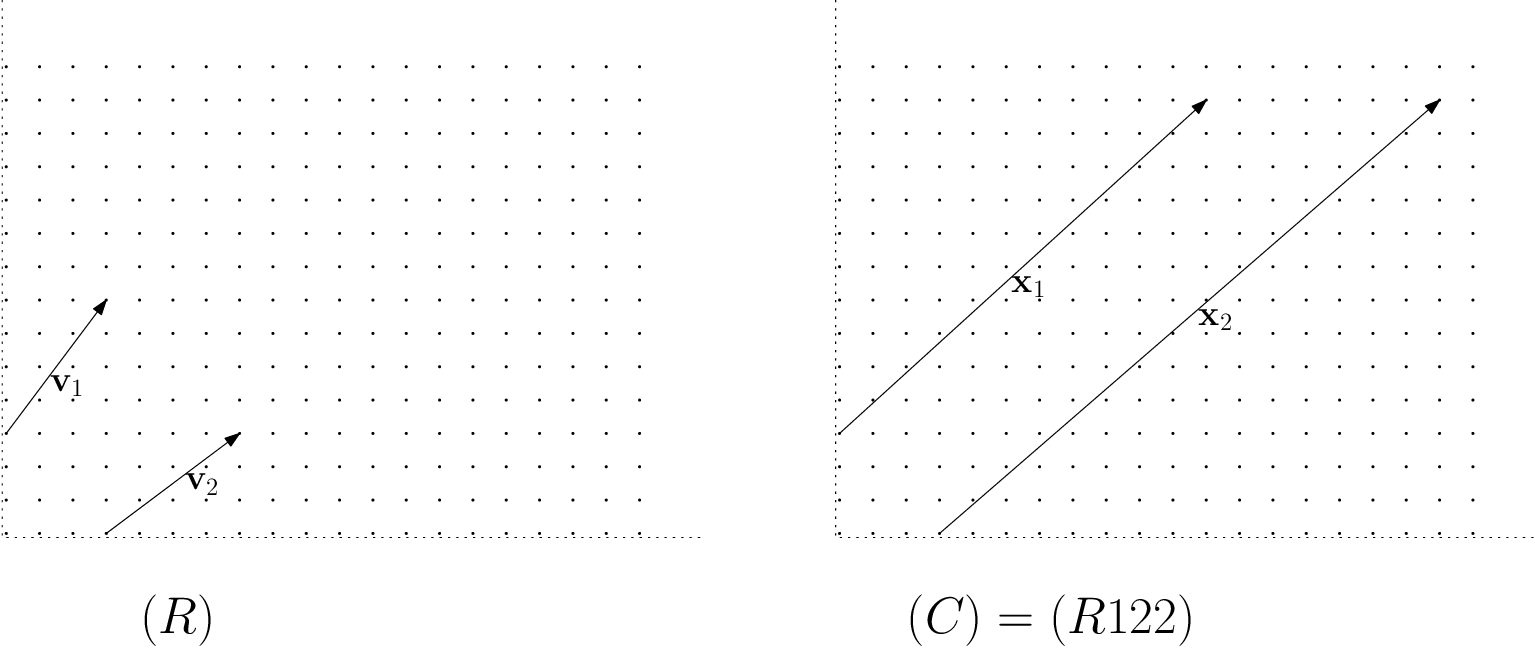}
\centering
\caption{The root $(R)$ (with vectors $\bv_1,\bv_2$) for the configuration $(C)=(R112)$ (with vectors $\bx_1,\bx_2$) of our initial GBS graph.}
\label{fig:ex7-configuration}
\end{figure}

In the subsequent \Cref{fig:ex7-angles}, we can observe concretely the effect of \Cref{prop:angles}. The picture is \emph{not} the affine representation; instead, we place all the vectors $\bv_1,\bv_2,\bx_1,\bx_2$ at the origin of the Euclidean plane $\bbR^2$, and we put in evidence the directions to which the vectors are pointing, and the \textit{angles} spanned by two of them. This shows explicitly how taking sons corresponds to making the angle narrower. The roots can ``see" more configurations, as they correspond to wider angles.

\begin{figure}[H]
\centering
\includegraphics[width=\textwidth]{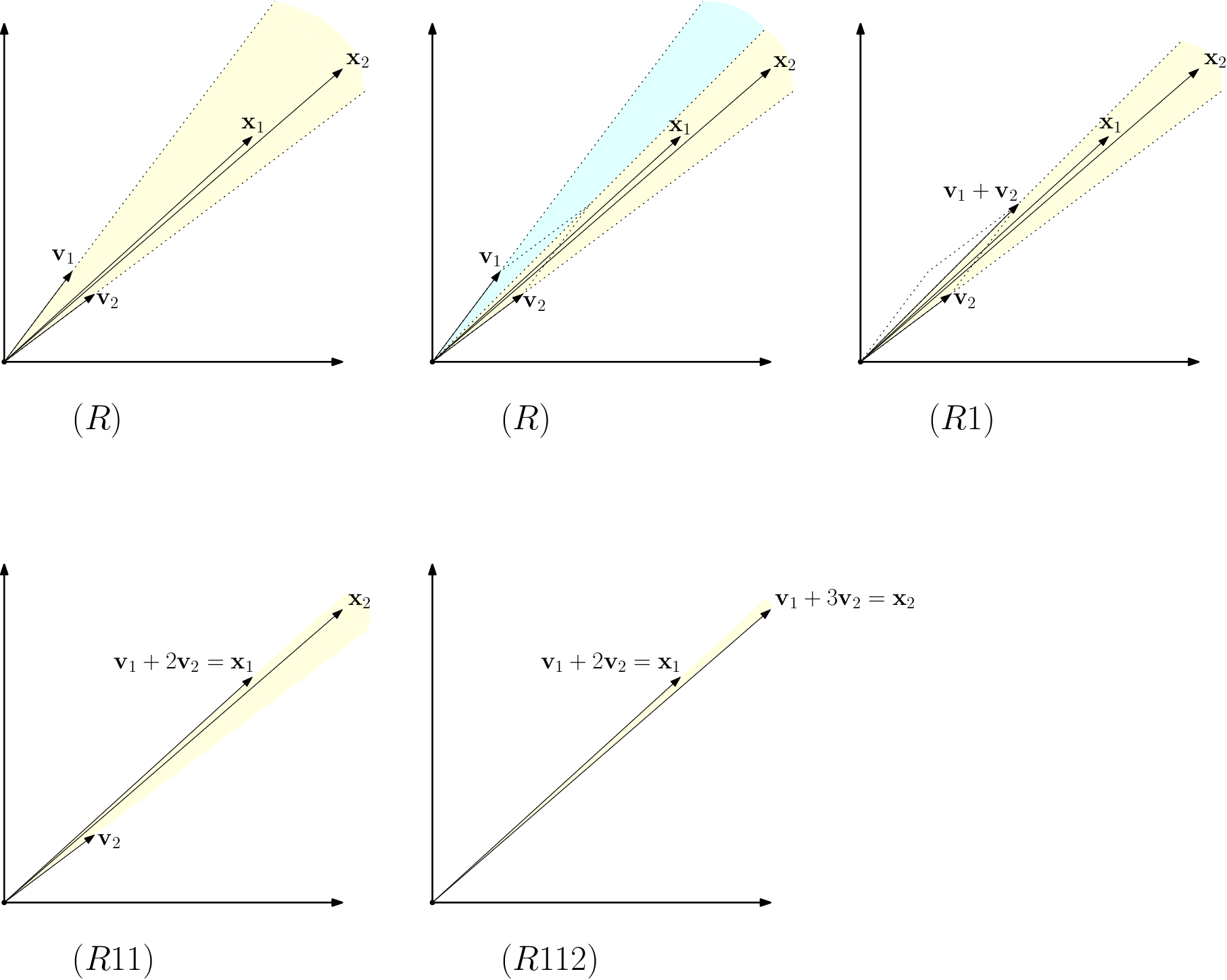}
\centering
\caption{In the picture, we see the vectors $\bv_1,\bv_2,\bx_1,\bx_2$ inside the Euclidean plane $\bbR^2$. We point out that this is NOT the affine representation; here we place the initial points of the vectors at the origin of the Euclidean plane, and we are mainly interested in the directions in which the vectors point. This should be thought as a graphical representation of the procedure described in the proof of \Cref{prop:angles}. \\
In the first image (top left), we see our initial vectors $\bv_1,\bv_2$ defining a certain angle, and the two ``target vectors" $\bx_1,\bx_2$ are contained inside this angle. In the second image (top middle), we draw a parallelogram, whose diagonal corresponds to the vector $\bv_1+\bv_2$, which divides the angle into two parts; the target vectors $\bx_1,\bx_2$ are contained in one of the two halves. In the third image (top right), we choose, among the two halves of the angle, the one containing the two target vectors $\bx_1,\bx_2$; this means that we have to choose the first son, changing $\bv_1$ into $\bv_1+\bv_2$ and keeping $\bv_2$ as it is. The fourth and the fifth image (bottom) show the subsequent iteration of the same procedure; at each step, the angle becomes narrower, but it keeps containing the two target vectors $\bx_1,\bx_2$. At the end of the procedure, we obtain exactly the desired configuration.}
\label{fig:ex7-angles}
\end{figure}

For the root $(R)$ we can explicitly compute the sequence $\dots,\bv_{-1},\bv_0,\bv_1,\bv_2,\bv_3,\dots$ as in \Cref{sec:sequence-of-roots}. This is infinite in both directions, and it is summarized in the following table:

\vspace{0.2cm}

\noindent
\begin{tabular}{p{3cm} | C{0.7cm} | C{0.7cm} | C{0.7cm} | C{0.7cm} | C{0.7cm} | C{0.7cm} | C{0.7cm} | C{0.7cm} | C{0.7cm} | C{0.7cm}}
    {\footnotesize vector} & \dots & $\bv_{-2}$ & $\bv_{-1}$ & $\bv_0$ & $\bv_1$ & $\bv_2$ & $\bv_3$ & $\bv_4$ & $\bv_5$ & \dots \\
    \hline
    {\footnotesize component $p=2$} & \dots & 9 & 7 & 5 & 3 & 4 & 9 & 14 & 19 & \dots \\
    \hline
    {\footnotesize component $p=3$} & \dots & 19 & 14 & 9 & 4 & 3 & 5 & 7 & 9 & \dots \\
    \hline
    {\footnotesize $k$ of the connection} & \dots & \scriptsize $k=2$ & \scriptsize $k=2$ & \scriptsize $k=2$ & \scriptsize $k=3$ & \scriptsize $k=3$ & \scriptsize $k=2$ & \scriptsize $k=2$ & \scriptsize $k=2$ &  \dots \\
    \hline
\end{tabular}

\vspace{0.2cm}

\noindent
In the first three rows we can see the list of the vectors, and their respective first components and second components. In the fourth row, we can see the parameter $k\ge2$ for the connection, satisfying $\bv_{i-1}+\bv_{i+1}=k\bv_i$ for $i\in\bbZ$. In order to compute the next term $\bv_{i+1}$ in the sequence, we just choose the minimum value of $k\ge2$ such that the next edge $\bv_{i+1}$ satisfies $\bv_{i+1}\ge\ba_1-\ba_2$; and similarly to extend the sequence on the left. The sequence is an arithmetic progression, except at the (finitely many) singular points where $k\not=2$ (and in fact, the reader can see the arithmetic progressions $4,9,14,19,\dots$ and $3,5,7,9,\dots$ in the table); the limit directions are $\bl^-=(2,5)$ and $\bl^+=(5,2)$. This is sufficient to describe the isomorphism problem for the given GBS group: the configurations that can be reached by means of slides, swaps, connections are exactly the iterated sons of the roots in the sequence above.

We are now interested about all configurations $(C)$ with minimal points $\ba_1,\ba_2$ and vectors $\by_1,\by_2$ such that $\gen{\by_1,\by_2}=\gen{(3,4),(4,3)}$ and $\by_1,\by_2$ is Nielsen equivalent to $(3,4),(4,3)$. By \Cref{prop:realizing-limit-directions} we know that all $\bl\in\gen{(3,4),(4,3)}$ which are part of a basis, all of whose components are strictly positive, and with at least one component $<3$ (because $\abs{\ba_1-\ba_2}=(3,3)$) can be realized as limit direction of some configuration. The values of $\bl$ satisfying these conditions are $(7\ell+6,1),(14\ell+5,2),(2,14\ell+5),(1,7\ell+6)$ for $\ell\ge0$ integer (see \Cref{thm:limit-directions-discrete}). In \Cref{fig:ex7-limit-angles} we can see all of these directions represented in the plane; as long as we stay far away from the axes, this is a discrete set. Between each consecutive pair of directions, lies exactly one isomorphism class of GBS groups, corresponding to the gray regions represented in \Cref{fig:ex7-limit-angles-2}. We point out that each gray region contains most of the configurations in a certain isomorphism class (to be precise, the full-tree ones) - but a few exceptional ones can fall outside. We also point out that there are a few isomorphism classes which do not contain any full-tree configuration, and thus are not represented by any region; these are described in \Cref{rmk:classification-roots-without-limit-directions}. We note that $\ba_1-\ba_2=(3,-3)$ belongs to $\gen{(3,4),(4,3)}$ but it is not part of a basis; thus in this case no twin roots can appear. For example, the following sequences of vectors (giving sequences of roots) correspond to different isomorphism classes of GBS groups.

\vspace{0.2cm}

\noindent
\begin{tabular}{p{3cm} | C{0.7cm} | C{0.7cm} || C{0.7cm} | C{0.7cm} | C{0.7cm} | C{0.7cm} | C{0.7cm} | C{0.7cm} | C{0.7cm} | C{0.7cm}}
    {\footnotesize vector} & $\bl^-$ &  & $\bv_0$ & $\bv_1$ & $\bv_2$ & $\bv_3$ & $\bv_4$ & \dots & & $\bl^+$ \\
    \hline
    {\footnotesize component $p=2$} & 5 &  & 4 & 5 & 11 & 17 & 23 & \dots & & 6 \\
    \hline
    {\footnotesize component $p=3$} & 2 &  & 3 & 2 & 3 & 4 & 5 & \dots & & 1 \\
    \hline
\end{tabular}

\vspace{0.3cm}

\noindent
\begin{tabular}{p{3cm} | C{0.7cm} | C{0.7cm} || C{0.7cm} | C{0.7cm} | C{0.7cm} | C{0.7cm} | C{0.7cm} | C{0.7cm} | C{0.7cm} | C{0.7cm}}
    {\footnotesize vector} & $\bl^-$ &  & $\bv_0$ & $\bv_1$ & $\bv_2$ & $\bv_3$ & $\bv_4$ & \dots & & $\bl^+$ \\
    \hline
    {\footnotesize component $p=2$} & 6 &  & 5 & 6 & 25 & 44 & 63 & \dots & & 19 \\
    \hline
    {\footnotesize component $p=3$} & 1 &  & 2 & 1 & 3 & 5 & 7 & \dots & & 2 \\
    \hline
\end{tabular}

\vspace{0.3cm}

\noindent
\begin{tabular}{p{3cm} | C{0.7cm} | C{0.7cm} || C{0.7cm} | C{0.7cm} | C{0.7cm} | C{0.7cm} | C{0.7cm} | C{0.7cm} | C{0.7cm} | C{0.7cm}}
    {\footnotesize vector} & $\bl^-$ &  & $\bv_0$ & $\bv_1$ & $\bv_2$ & $\bv_3$ & $\bv_4$ & \dots & & $\bl^+$ \\
    \hline
    {\footnotesize component $p=2$} & 19 &  & 6 & 19 & 32 & 54 & 58 & \dots & & 13 \\
    \hline
    {\footnotesize component $p=3$} & 2 &  & 1 & 2 & 3 & 4 & 5 & \dots & & 1 \\
    \hline
\end{tabular}

\vspace{0.3cm}

\noindent
\begin{tabular}{p{3cm} | C{0.7cm} | C{0.7cm} || C{0.7cm} | C{0.7cm} | C{0.7cm} | C{0.7cm} | C{0.7cm} | C{0.7cm} | C{0.7cm} | C{0.7cm}}
    {\footnotesize vector} & $\bl^-$ &  & $\bv_0$ & $\bv_1$ & $\bv_2$ & $\bv_3$ & $\bv_4$ & \dots & & $\bl^+$ \\
    \hline
    {\footnotesize component $p=2$} & 13 &  & 6 & 13 & 46 & 79 & 112 & \dots & & 33 \\
    \hline
    {\footnotesize component $p=3$} & 1 &  & 1 & 1 & 3 & 5 & 7 & \dots & & 2 \\
    \hline
\end{tabular}

\vspace{0.2cm}

\end{ex}

\begin{figure}[H]
\centering
\includegraphics[width=0.55\textwidth]{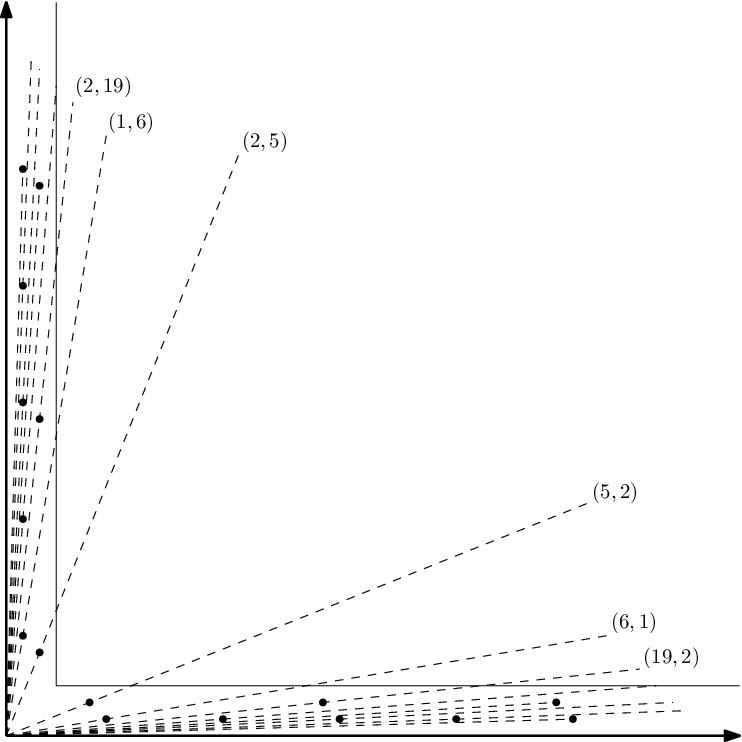}
\centering
\caption{The set of all possible limit directions for the subgroup $H=\gen{(4,3),(3,4)}$ as in \Cref{ex7}. The dots represents all the elements which are part of a basis for $H$, and which lie near enough to the two axes - we only consider elements which lie left/below of the line in the figure. For each dot, we draw the corresponding direction (the dashed lines), which can always be realized as limit direction of some GBS group. From the picture we can see that the points form a finite union of arithmetic progressions, accumulating only in the horizontal and vertical directions.}
\label{fig:ex7-limit-angles}
\end{figure}

\begin{figure}[H]
\centering
\includegraphics[width=0.55\textwidth]{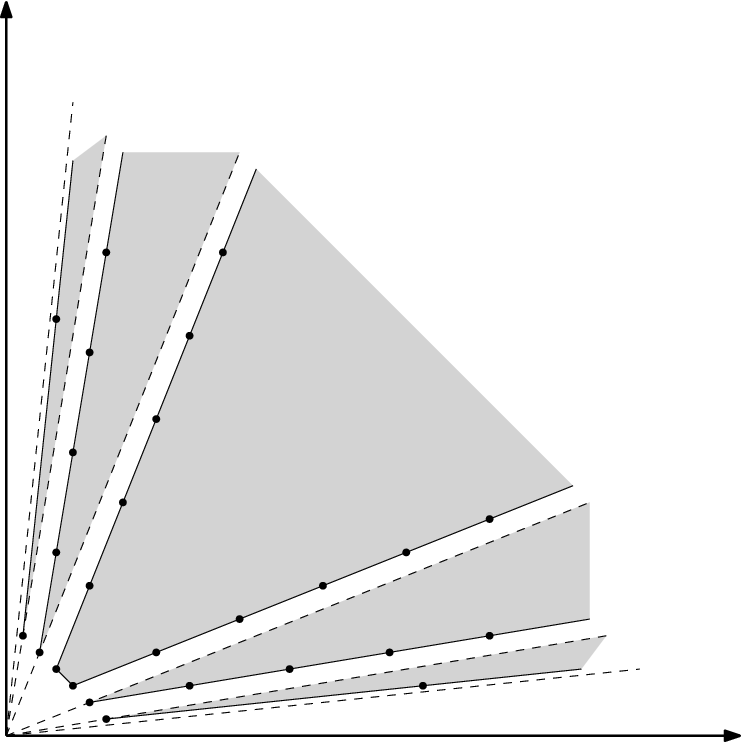}
\centering
\caption{The limit angles obtained from the limit directions represented in \Cref{fig:ex7-limit-angles}. Each gray region is an isomorphism class of GBS groups. Two consecutive points on the boundary of a gray region give us one of the roots for the binary trees - consecutive pairs of points being roots related by a connection move.}
\label{fig:ex7-limit-angles-2}
\end{figure}

\begin{ex}\label{ex11}
Consider the GBS graph $(\Gamma,\psi)$ with one vertex and two edges, as in \Cref{fig:ex11}. We call $\Lambda$ its affine representation, given by edges
\[
    \begin{cases}
        \ba_1\edge\ba_1+4\bz\\
        \ba_2\edge\ba_2+7\bz
    \end{cases}
\]
where $\ba_1=(0,5)$ and $\ba_2=(3,0)$ and $\bz=(2,1)$. Note that the subgroup generated by the two vectors is $\gen{4\bz,7\bz}=\gen{\bz}\cong\bbZ$. 

\begin{figure}[H]
\centering
\includegraphics[width=0.9\textwidth]{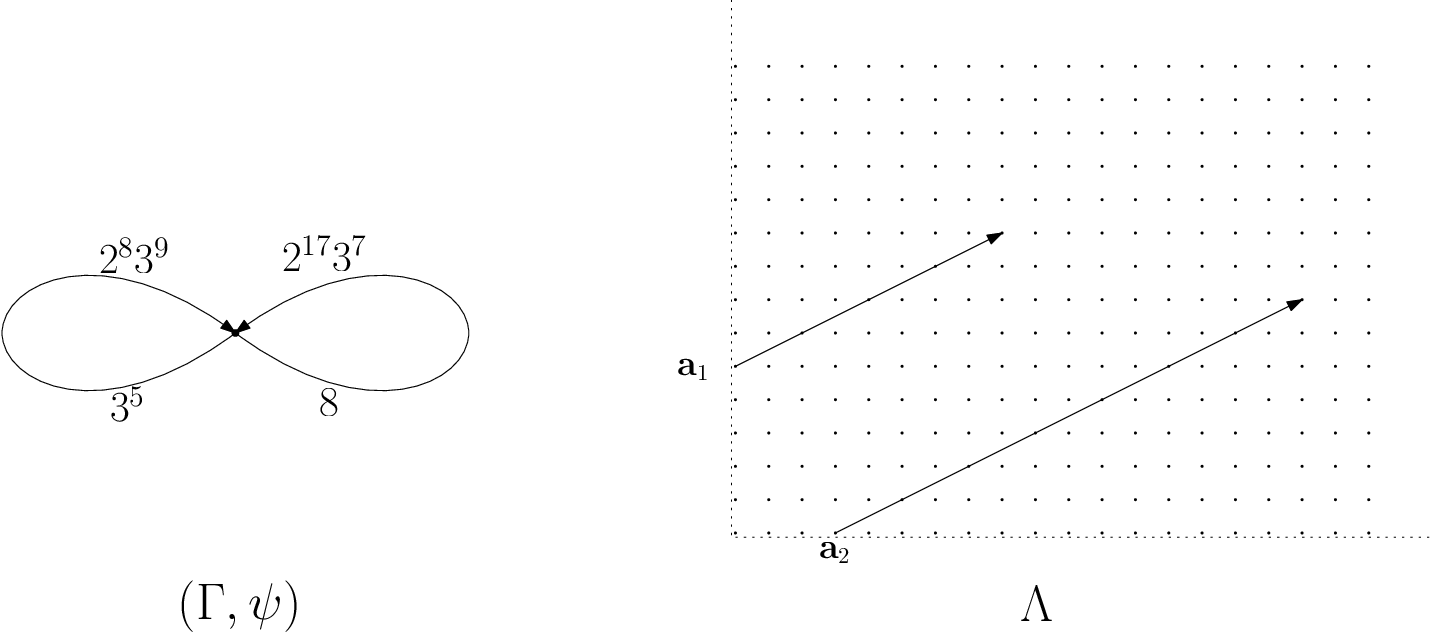}
\centering
\caption{The GBS graph $(\Gamma,\psi)$ of \Cref{ex11} and the corresponding affine representation $\Lambda$.}
\label{fig:ex11}
\end{figure}

It is fairly easy to compute the associated sequence of roots, which is summarized in the following table:

\begin{center}
\begin{tabular}{ || C{0.7cm} | C{0.7cm} | C{0.7cm} | C{0.7cm} | C{0.7cm} | C{0.7cm} | C{0.7cm} | C{0.7cm}}
    $\bv_{-1}$ & $\bv_0$ & $\bv_1$ & $\bv_2$ & $\bv_3$ & $\bv_4$ & $\bv_5$ & \dots \\
    \hline
    $5\bz$ & $4\bz$ & $7\bz$ & $10\bz$ & $13\bz$ & $16\bz$ & $19\bz$ & \dots \\
    \hline
\end{tabular}
\end{center}

\vspace{0.2cm}

We now consider all the configurations of the form
\[
    (C)=\begin{cases}
        \ba_1\edge\ba_1+k_1\bz\\
        \ba_2\edge\ba_2+k_2\bz
    \end{cases}
\]
for $k_1,k_2\ge1$ integers with $(k_1,k_2)=1$, and we want to divide them into isomorphism classes, as in \Cref{cor:rank-1-limit-directions-discrete}. The configuration $(C)$ is full-tree if and only if $k_1\ge2$ and $k_2\ge5$. The configuration $(C)$ is root if and only if $k_1\le k_2+1$ and $k_2\le k_1+4$. Thus, if we want to obtain a sequence of roots infinite on the left, then it eventually be an arithmetic progression of ratio between $-1$ and $+4$ (and with consecutive terms of the progression being coprime). 

There is a unique arithmetic progression of ratio $-1$, giving the sequence of roots

\begin{center}
\begin{tabular}{ C{0.7cm} | C{0.7cm} | C{0.7cm} | C{0.7cm} | C{0.7cm} | C{0.7cm} | C{0.7cm} | C{0.7cm} | C{0.7cm} | C{0.7cm}}
    \dots & $\bv_{-2}$ & $\bv_{-1}$ & $\bv_0$ & $\bv_1$ & $\bv_2$ & $\bv_3$ & $\bv_4$ & $\bv_5$ & \dots \\
    \hline
    \dots & $8\bz$ & $7\bz$ & $6\bz$ & $5\bz$ & $9\bz$ & $13\bz$ & $17\bz$ & $21\bz$ & \dots \\
    \hline
\end{tabular}
\end{center}

There is a unique arithmetic progression of ratio $+1$, giving the sequence of roots

\begin{center}
\begin{tabular}{ || C{0.7cm} | C{0.7cm} | C{0.7cm} | C{0.7cm} | C{0.7cm} | C{0.7cm} | C{0.7cm} | C{0.7cm}}
    $\bv_{-1}$ & $\bv_0$ & $\bv_1$ & $\bv_2$ & $\bv_3$ & $\bv_4$ & $\bv_5$ & \dots \\
    \hline
    $3\bz$ & $4\bz$ & $5\bz$ & $6\bz$ & $7\bz$ & $8\bz$ & $9\bz$ & \dots \\
    \hline
\end{tabular}
\end{center}

There is a unique arithmetic progression of ratio $+2$, giving the sequence of roots

\begin{center}
\begin{tabular}{ || C{0.7cm} | C{0.7cm} | C{0.7cm} | C{0.7cm} | C{0.7cm} | C{0.7cm} | C{0.7cm} | C{0.7cm}}
    $\bv_{-1}$ & $\bv_0$ & $\bv_1$ & $\bv_2$ & $\bv_3$ & $\bv_4$ & $\bv_5$ & \dots \\
    \hline
    $4\bz$ & $3\bz$ & $5\bz$ & $7\bz$ & $9\bz$ & $11\bz$ & $13\bz$ & \dots \\
    \hline
\end{tabular}
\end{center}

There are exactly two arithmetic progression of ratio $+3$, giving the sequences of roots

\begin{center}
\begin{tabular}{ || C{0.7cm} | C{0.7cm} | C{0.7cm} | C{0.7cm} | C{0.7cm} | C{0.7cm} | C{0.7cm} | C{0.7cm}}
    $\bv_{-1}$ & $\bv_0$ & $\bv_1$ & $\bv_2$ & $\bv_3$ & $\bv_4$ & $\bv_5$ & \dots \\
    \hline
    $5\bz$ & $4\bz$ & $7\bz$ & $10\bz$ & $13\bz$ & $16\bz$ & $19\bz$ & \dots \\
    \hline
\end{tabular}
\end{center}

\begin{center}
\begin{tabular}{ || C{0.7cm} | C{0.7cm} | C{0.7cm} | C{0.7cm} | C{0.7cm} | C{0.7cm} | C{0.7cm} | C{0.7cm}}
    $\bv_{-1}$ & $\bv_0$ & $\bv_1$ & $\bv_2$ & $\bv_3$ & $\bv_4$ & $\bv_5$ & \dots \\
    \hline
    $3\bz$ & $2\bz$ & $5\bz$ & $8\bz$ & $11\bz$ & $14\bz$ & $17\bz$ & \dots \\
    \hline
\end{tabular}
\end{center}

There are exactly two arithmetic progression of ratio $+4$. One of them has the same sequence of roots as the progression of ratio $-1$. The other gives the sequence of roots

\begin{center}
\begin{tabular}{ || C{0.7cm} | C{0.7cm} | C{0.7cm} | C{0.7cm} | C{0.7cm} | C{0.7cm} | C{0.7cm} | C{0.7cm}}
    $\bv_{-1}$ & $\bv_0$ & $\bv_1$ & $\bv_2$ & $\bv_3$ & $\bv_4$ & $\bv_5$ & \dots \\
    \hline
    $2\bz$ & $3\bz$ & $7\bz$ & $11\bz$ & $15\bz$ & $19\bz$ & $23\bz$ & \dots \\
    \hline
\end{tabular}
\end{center}

This gives a total of $6$ isomorphism classes of GBSs. There are also other isomorphism classes, but their isomorphism problems do not contain any full-tree configuration (and so are much easier to describe).
\end{ex}

\bibliographystyle{alpha}

\footnotesize
\begin{thebibliography}{ACRK25b}
	
	\bibitem[ACRK25a]{ACK-out}
	D.~Ascari, M.~Casals-Ruiz, and I.~Kazachkov.
	\newblock {Automorphisms of graphs of groups}, 2025.
	\newblock preprint.
	
	\bibitem[ACRK25b]{ACK-iso1}
	D.~Ascari, M.~Casals-Ruiz, and I.~Kazachkov.
	\newblock {On the isomorphism problem for cyclic JSJ decompositions: vertex elimination}, 2025.
	\newblock preprint.
	
	\bibitem[ACRK25c]{ACK-iso2}
	D.~Ascari, M.~Casals-Ruiz, and I.~Kazachkov.
	\newblock {On the isomorphism problem for generalized Baumslag-Solitar groups: invariants and flexible configurations}, 2025.
	\newblock preprint.
	
	\bibitem[CF08]{CF08}
	M.~Clay and M.~Forester.
	\newblock On the isomorphism problem for generalized {B}aumslag-{S}olitar groups.
	\newblock {\em Algebr. Geom. Topol.}, 8(4):2289--2322, 2008.
	
	\bibitem[CRKZ21]{CRKZ21}
	M.~Casals-Ruiz, I.~Kazachkov, and A.~Zakharov.
	\newblock Commensurability of {B}aumslag-{S}olitar groups.
	\newblock {\em Indiana Univ. Math. J.}, 70(6):2527--2555, 2021.
	
	\bibitem[DG11]{DG11}
	F.~Dahmani and V.~Guirardel.
	\newblock The isomorphism problem for all hyperbolic groups.
	\newblock {\em Geom. Funct. Anal.}, 21(2):223--300, 2011.
	
	\bibitem[DT19]{DT19}
	F.~Dahmani and N.~Touikan.
	\newblock Deciding isomorphy using {D}ehn fillings, the splitting case.
	\newblock {\em Invent. Math.}, 215(1):81--169, 2019.
	
	\bibitem[Dud17]{Dud17}
	F.~A. Dudkin.
	\newblock The isomorphism problem for generalized {B}aumslag-{S}olitar groups with one mobile edge.
	\newblock {\em Algebra Logika}, 56(3):300--316, 2017.
	
	\bibitem[For06]{For06}
	M.~Forester.
	\newblock Splittings of generalized {B}aumslag-{S}olitar groups.
	\newblock {\em Geom. Dedicata}, 121:43--59, 2006.
	
	\bibitem[Lev07]{Lev07}
	G.~Levitt.
	\newblock On the automorphism group of generalized {B}aumslag-{S}olitar groups.
	\newblock {\em Geom. Topol.}, 11:473--515, 2007.
	
	\bibitem[Sel95]{Sel95}
	Z.~Sela.
	\newblock The isomorphism problem for hyperbolic groups. {I}.
	\newblock {\em Ann. of Math. (2)}, 141(2):217--283, 1995.
	
	\bibitem[Ser77]{Ser77}
	J.~P. Serre.
	\newblock {\em Arbres, amalgames, {${\rm SL}\sb{2}$}}, volume No. 46 of {\em Ast\'erisque}.
	\newblock Soci\'et\'e{} Math\'ematique de France, Paris, 1977.
	\newblock Avec un sommaire anglais, R\'edig\'e{} avec la collaboration de Hyman Bass.
	
	\bibitem[Wan25]{Wan25}
	D.~Wang.
	\newblock The isomorphism problem for small-rose generalized {B}aumslag-{S}olitar groups.
	\newblock {\em J. Algebra}, 661:193--217, 2025.
	
\end{thebibliography}

\end{document}